\documentclass[a4paper,12pt]{amsart}
\usepackage[english]{babel}
\usepackage[T1]{fontenc}
\usepackage[left=1in,right=1in,top=1.2in,bottom=1.2in]{geometry}
\usepackage{times}
\usepackage{microtype}

\usepackage{commath,amssymb,amscd,mathrsfs,mathtools,dsfont,bbm,multirow,array,caption,comment,pifont}
\usepackage[all]{xy}
\usepackage{enumitem}
\usepackage{longtable}

\usepackage{cite}			

\allowdisplaybreaks

\usepackage[usenames,dvipsnames,svgnames,table]{xcolor}
\definecolor{red}{RGB}{255,25,25}
\definecolor{blue}{RGB}{25,50,200}
\usepackage{tikz-cd}

\usepackage{thmtools}   

\usepackage[pagebackref,linktocpage]{hyperref}

\hypersetup{
pdftitle={},				
pdfauthor={Fei Hu},	
colorlinks=true,			
linkcolor=red,			
citecolor=MidnightBlue,	
filecolor=magenta,		
urlcolor=MidnightBlue	
}

\usepackage{cleveref}	

\newtheorem{theorem}{Theorem}[section]
\crefname{theorem}{Theorem}{Theorems}
\newtheorem{lemma}[theorem]{Lemma}
\crefname{lemma}{Lemma}{Lemmas}
\newtheorem{proposition}[theorem]{Proposition}
\crefname{proposition}{Proposition}{Propositions}

\crefname{prop}{Proposition}{Propositions}
\newtheorem{corollary}[theorem]{Corollary}
\crefname{corollary}{Corollary}{Corollaries}

\crefname{cor}{Corollary}{Corollaries}

\crefname{conjecture}{Conjecture}{Conjectures}

\crefname{conj}{Conjecture}{Conjectures}
\newtheorem*{conj*}{Conjecture}
\crefname{conj}{Conjecture}{Conjectures}

\crefname{conjA}{Conjecture}{Conjecture}

\crefname{conjB}{Conjecture}{Conjecture}

\crefname{conjC}{Conjecture}{Conjecture}

\crefname{conjDk}{Conjecture}{Conjecture}

\crefname{conjD}{Conjecture}{Conjecture}

\crefname{conjH}{Conjecture}{Conjecture}

\crefname{conjGr}{Conjecture}{Conjecture}

\theoremstyle{definition}
\newtheorem{definition}[theorem]{Definition}
\crefname{definition}{Definition}{Definitions}

\crefname{defn}{Definition}{Definitions}
\newtheorem{example}[theorem]{Example}
\crefname{example}{Example}{Examples}

\crefname{notation}{Notation}{Notation}
\newtheorem*{notation*}{Notation}
\crefname{notation}{Notation}{Notation}
\newtheorem*{convention*}{Convention}
\crefname{convention}{Convention}{Convention}

\crefname{problem}{Problem}{Problems}
\newtheorem{question}[theorem]{Question}
\crefname{question}{Question}{Questions}

\crefname{condition}{Condition}{Conditions}
\newtheorem{assumption}[theorem]{Assumption}
\crefname{assumption}{Assumption}{Assumptions}

\crefname{propGr}{Property}{Property}

\theoremstyle{remark}

\crefname{rmk}{Remark}{Remarks}
\newtheorem*{rmk*}{Remark}
\crefname{rmk}{Remark}{Remarks}
\newtheorem{remark}[theorem]{Remark}
\crefname{remark}{Remark}{Remarks}

\crefname{fact}{Fact}{Facts}
\newtheorem{claim}[theorem]{Claim}
\crefname{claim}{Claim}{Claims}
\newtheorem*{claim*}{Claim}
\crefname{claim}{Claim}{Claims}

\crefname{step}{Step}{Steps}

\crefname{case}{Case}{Cases}

\numberwithin{equation}{section}

\newcommand{\longsurjmap}{\relbar\joinrel\twoheadrightarrow}
\newcommand{\lra}{\longrightarrow}

\newcommand{\doi}[1]{\href{https://doi.org/#1}{{\tt DOI:#1}}}

\newcommand{\bC}{\mathbf{C}}

\newcommand{\bN}{\mathbf{N}}

\newcommand{\bP}{\mathbf{P}}
\newcommand{\bQ}{\mathbf{Q}}
\newcommand{\bR}{\mathbf{R}}

\newcommand{\bZ}{\mathbf{Z}}

\newcommand{\bk}{\mathbf{k}}

\newcommand{\sE}{\mathscr{E}}

\newcommand{\sL}{\mathscr{L}}

\newcommand{\sT}{\mathsf{T}}

\newcommand{\alg}{\operatorname{alg}}

\newcommand{\CH}{\mathsf{CH}}

\newcommand{\et}{{\textrm{\'et}}}

\newcommand{\GK}{\operatorname{GKdim}}

\newcommand{\Hom}{\operatorname{Hom}}
\newcommand{\id}{\operatorname{id}}

\newcommand{\isom}{\simeq}

\newcommand{\lov}{\operatorname{lov}}
\newcommand{\Mat}{\operatorname{M}}

\newcommand{\Nef}{\operatorname{Nef}}
\newcommand{\N}{\mathsf{N}}

\newcommand{\NP}{\mathsf{NP}}
\newcommand{\NPbar}{\overline{\mathsf{NP}}}

\newcommand{\num}{\equiv}
\newcommand{\wnum}{\equiv_{\mathsf{w}}}
\newcommand{\wg}{\succ_{\mathsf{w}}}
\newcommand{\wge}{\succeq_{\mathsf{w}}}

\newcommand{\plov}{\operatorname{plov}}

\newcommand{\pr}{\operatorname{pr}}

\newcommand{\rank}{\operatorname{rank}}

\newcommand{\Sym}{\operatorname{Sym}}

\newcommand{\Vol}{\operatorname{Vol}}
\newcommand{\W}{\mathsf{W}}
\newcommand{\Z}{\mathsf{Z}}

\begin{document}

\title[An upper bound for $\plov$]{An upper bound for polynomial volume growth of automorphisms of zero entropy}

\author{Fei Hu}
\address{
School of Mathematics, Nanjing University, Nanjing, China
\endgraf
Department of Mathematics, University of Oslo, Oslo, Norway
\endgraf
Department of Mathematics, Harvard University, Cambridge, USA}
\email{\href{mailto:fhu@nju.edu.cn}{\tt fhu@nju.edu.cn}}

\author{Chen Jiang}
\address{Shanghai Center for Mathematical Sciences \& School of Mathematical Sciences, Fudan University, Shanghai, China}
\email{\href{mailto:chenjiang@fudan.edu.cn}{chenjiang@fudan.edu.cn}}
 
\begin{abstract}
Let $X$ be a normal projective variety of dimension $d$ over an algebraically closed field and $f$ an automorphism of $X$.
Suppose that the pullback $f^*|_{\mathsf{N}^1(X)_\mathbf{R}}$ of $f$ on the real N\'eron--Severi space $\mathsf{N}^1(X)_\mathbf{R}$ is unipotent and denote the index of the eigenvalue $1$ by $k+1$.
We establish the following upper bound for the polynomial volume growth $\mathrm{plov}(f)$ of $f$:
\[
\mathrm{plov}(f) \le (k/2 + 1)d.
\]
This inequality is optimal in certain cases.
Moreover, we prove that $k\le 2(d-1)$, extending a result of Dinh--Lin--Oguiso--Zhang for compact K\"ahler manifolds to arbitrary characteristic.
By combining these two inequalities, we obtain the optimal bound
\[
\mathrm{plov}(f) \le d^2,
\]
that affirmatively answers the questions of Cantat--Paris-Romaskevich and Lin--Oguiso--Zhang.
\end{abstract}



\thanks{The first author was supported by an NSFC grant \#12371045 and Young Research Talents grant \#300814 from the Research Council of Norway.
The second author was supported by National Key Research and Development Program of China \#2023YFA1010600, NSFC for Innovative Research Groups \#12121001, and National Key Research and Development Program of China \#2020YFA0713200.}

\maketitle

\section{Introduction}

Given a surjective endomorphism $f$ of a smooth projective variety $X$ of dimension $d$ over $\bC$,
Gromov \cite{Gromov03} introduced in 1977 the so-called {\it iterated graph} $\Gamma_n \subset X^n$ of $f$, i.e., the graph of the morphism $(f,\ldots,f^{n-1}) \colon X \to X^{n-1}$.
He bounded the topological entropy $h_{\mathrm{top}}(f)$ of $f$ by the {\it volume growth $\lov(f)$} of $f$, and further by the {\it algebraic entropy $h_{\mathrm{alg}}(f)$} of $f$ as follows:
\begin{align*}
h_{\mathrm{top}}(f) &\le \lov(f) \coloneqq \limsup_{n\to \infty} \frac{\log \Vol (\Gamma_n)}{n} \\
&\le h_{\alg}(f) \coloneqq \log \max_{0\le i\le d} \lambda_i(f),
\end{align*}
where the volume $\Vol(\Gamma_n)$ is computed against the ample divisor on the product variety $X^n$ induced from an arbitrary ample divisor $H_X$ on $X$, 
and the {\it $i$-th dynamical degree $\lambda_i(f)$} of $f$ is defined by
\begin{align}\label{eq:dyndeg}
\lambda_i(f) &\coloneqq \lim_{n\to \infty} ((f^n)^*H_X^i \cdot H_X^{d-i})^{1/n}.
\end{align}
By combining this with Yomdin's remarkable inequality $h_{\mathrm{alg}}(f) \le h_{\mathrm{top}}(f)$---which resolves Shub's entropy conjecture (see \cite{Yomdin87})---Gromov's result yields the fundamental equality in higher-dimensional algebraic and holomorphic dynamics:
\[
h_{\mathrm{top}}(f) = \lov(f) = h_{\mathrm{alg}}(f).
\]
Recently, there has been growing interest in varieties with slow dynamics (see, e.g., \cite{Labrousse12, Marco13, Cantat18-ICM, LB19, CPR21, FFO21, DLOZ22, LOZ, Hu-GK-AV}), particularly in automorphisms of zero entropy.

From now on, we work over an algebraically closed field $\bk$ of arbitrary characteristic, unless otherwise specified.
Let $X$ be a normal projective variety of dimension $d$, $H_X$ an ample divisor on $X$, and $f$ an automorphism of $X$, over $\bk$.

Denote by $\N^1(X)_\bR$ the real N\'eron--Severi space of Cartier divisors on $X$ modulo numerical equivalence, which is a finite-dimensional $\bR$-vector space.
Suppose that $f$ is of zero entropy.
Equivalently, the induced pullback action $f^*|_{\N^1(X)_\bR}$ is {\it quasi-unipotent}, i.e., all eigenvalues of $f^*|_{\N^1(X)_\bR}$ are roots of unity (see, e.g., \cite[Lemma~2.2]{Hu-GK-AV}).
Denote by $k+1$ the maximum size of Jordan blocks in the Jordan canonical form of $f^*|_{\N^1(X)_\bR}$.
By \cite[Lemma~6.12]{Keeler00}, we note that $k$ is always an even nonnegative integer, say $k=2r$.

\subsection{Polynomial volume growth}
\label{subsec plov}

In their study of polynomial entropy in slow dynamics, Cantat and Paris-Romaskevich \cite{CPR21} introduced the \textit{polynomial volume growth $\plov(f)$} of $f$ as follows:
\[
\plov(f) \coloneqq \limsup_{n\to \infty} \frac{\log \Vol(\Gamma_n)}{\log n}.
\]
This quantity is closely related to the Gelfand--Kirillov dimension of the twisted homogeneous coordinate ring associated with $(X,f)$ (see \cite[Proposition~6.11]{Keeler00}).
Lin, Oguiso, and Zhang \cite{LOZ} first noticed the coincidence of these two invariants in algebraic dynamics and noncommutative geometry, where, among many other things, they also improved Keeler's upper bound $k(d-1)+d$ for $\plov(f)$, using dynamical filtrations introduced by their earlier joint work \cite{DLOZ22} with Dinh.

Our first main result establishes a new upper bound for $\plov(f)$, which is nearly half of the previously known upper bounds.
This result provides an affirmative answer to \cite[Question~6.6]{LOZ} and \cite[Question~2.10]{Hu-GK-AV}.
The proof relies on combinatorial and representation-theoretic techniques (see \cref{rmk:idea-plov}).
Along the way, we also recover the unimodality of restricted partition numbers (see \S\ref{subsec rank}).

\begin{theorem}
\label{thm:plov}
Let $X$ be a normal projective variety of dimension $d$ over $\bk$ and $f$ an automorphism of $X$.
Suppose that $f^*|_{\N^1(X)_\bR}$ is quasi-unipotent and the maximum size of Jordan blocks of $f^*|_{\N^1(X)_\bR}$ is $k+1$.
Then we have
\[
\plov(f) \le (k/2+1)d.
\]
\end{theorem}

The upper bound in \cref{thm:plov} is optimal when $k/2+1$ divides $d$ (see \cref{ex optimal}).
On the other hand, if $k=2$ and $d$ is odd, then the actual optimal upper bound turns out to be $2d-1$ as shown in \cref{cor:k=2} (see also \cref{prop improve -1}).

Later, in \S\ref{subsec degree seq}, we shall prove that $k\leq 2d-2$ (see \cref{cor:degree-growth-upper-bound}), which generalizes \cite[Theorem~1.1]{DLOZ22} to arbitrary characteristic.
Thus, by combining \cref{thm:plov} and \cref{cor:degree-growth-upper-bound}, we also provide affirmative answers to \cite[Question~4.1]{CPR21} and \cite[Question~1.5(1)]{LOZ} in arbitrary characteristic (see \cref{rmk:Kahler-singular}).

\begin{corollary}
\label{cor:plov}
Let $X$ be a normal projective variety of dimension $d$ over $\bk$.
Let $f$ be an automorphism of zero entropy of $X$.
Then one has
\[
\plov(f) \le d^2.
\]
This upper bound is optimal by \cref{ex optimal}.
\end{corollary}

\begin{remark}
In noncommutative geometry, Artin, Tate, and Van den Bergh \cite{ATVdB90,AVdB90} introduced in the 1990s the so-called twisted homogeneous coordinate ring $B\coloneqq B(X,f,\sL)$ associated with a normal projective variety $X$ over $\bk$, an automorphism $f$ of $X$, and an invertible sheaf $\sL$ on $X$.
Keeler \cite{Keeler00} proved that the Gelfand--Kirillov dimension $\GK(B)$ is finite (and equals $\plov(f)+1$) if and only if the restriction $f^*|_{\N^1(X)_\bR}$ is quasi-unipotent (see also \cite{LOZ}).
Consequently, our result also provides an optimal upper bound $d^2+1$ for $\GK(B)$.
\end{remark}

\begin{remark}
\label{rmk:idea-plov}
The idea of our proof of \cref{thm:plov} is simple and traces back to \cite{AVdB90,Keeler00}.
Precisely, without loss of generality, we may assume that $f^*|_{\N^1(X)_\bR}$ is unipotent and hence can be written as $f^*|_{\N^1(X)_\bR} = \id +N$, where $N$ is a nilpotent operator on $\N^1(X)_\bR$.
Then by \cite[Proof of Lemma~6.13]{Keeler00} or \cite[Lemma~2.16]{LOZ} (see also \cref{lemma:plov}), we will be interested in the vanishing of intersection numbers $N^{i_1}H_X \cdots N^{i_d} H_X$.
Applying the projection formula and proceeding by induction, we derive a homogeneous system of linear equations, where these intersection numbers are unknowns.

To establish the vanishing of certain intersection numbers, we prove that the matrix of coefficients in the above homogeneous system of linear equations, denoted by $A_{k,d,n}$ later, is of full column rank whenever $n>dk/2$.
Although analyzing a single matrix $A_{k,d,n}$ directly appears to be quite intricate, our key observation is that the family of such matrices (as $n$ varies) can be naturally realized as representative matrices of Lefschetz operators of a representation of the Lie algebra $\mathfrak{sl}_2(\bC)$.
This realization enables us to apply a hard Lefschetz-type result, which asserts that the product matrix $A_{k,d,n+1}A_{k,d,n+2}\cdots A_{k,d,dk-n}$ is invertible for any $n<dk/2$.
See \cref{thm:full-rank} for details.
\end{remark}

\subsection{Polynomial growth rate of degree sequences}
\label{subsec degree seq}

The dynamical degrees $\lambda_i(f)$ of $f$, defined by \cref{eq:dyndeg}, are fundamental invariants of the algebraic dynamical system $(X, f)$, and their definition remains valid in arbitrary characteristic.
Roughly speaking, these invariants capture the exponential growth rate of the degree sequences $(\deg_i(f^n))_{n\in\bN}$, where
\[
\deg_i(f^n)\coloneqq (f^n)^*H_X^i\cdot H_X^{d-i}.
\]
Note that while $\deg_i(f^n)$ depends on the choice of $H_X$, its growth rate does not.
For further details, see \cite{DS17, Cantat18-ICM, Truong20, Dang20, CX20, DF21, HT, Xie-SC} and references therein.

Suppose that $f$ is of zero entropy as in \S\ref{subsec plov}, i.e., $f^*|_{\N^1(X)_\bR}$ is quasi-unipotent.
Then all $\lambda_i(f)$ are equal to $1$ by the log-concavity property of dynamical degrees (see \cite[Lemma~5.7]{Truong20}).
Thus, in this case, we are more interested in the polynomial growth rate of degree sequences $(\deg_i(f^n))_{n\in\bN}$.
When $X$ is a smooth complex projective variety, Dinh, Lin, Oguiso, and Zhang \cite{DLOZ22} obtained the following optimal estimate:
\[
\deg_1(f^n) = O(n^{2(d-1)}), \text{ as } n\to \infty,
\]
extending earlier results of Gizatullin \cite{Gizatullin80}, Cantat \cite{Cantat14a}, and Lo Bianco \cite{LB19} to higher dimensions.
A key ingredient of their proof is the use of dynamical filtrations; see \cite[\S3]{DLOZ22}.

Our next main result establishes the positivity (and hence the nonvanishing) of certain intersection numbers, which can be viewed as the opposite of \cref{thm:plov}.
We then derive an upper bound on the polynomial growth rate of the degree sequence $(\deg_1(f^n))_{n\in \bN}$, which extends \cite[Theorem~1.1]{DLOZ22} to arbitrary characteristic, without using the theory of dynamical filtrations developed in \cite[\S3]{DLOZ22}.

\begin{theorem}[{cf. \cref{thm:full-version-positivity}}]
\label{thm:positivity}
Let $X$ be a normal projective variety of dimension $d$, $H_X$ an ample divisor on $X$, and $f$ an automorphism of $X$, over $\bk$.
Suppose that $f^*|_{\N^1(X)_\bR}=\id+N$ is unipotent with $N^{2r}\neq 0$ and $N^{2r+1} = 0$, where $r$ is a positive integer.
Then there are positive integers $t_r,\ldots,t_1$ with $\sum_{i=1}^{r} t_i <d$, such that for any $0\le j\le r-1$, we have that
\[
(N^{2r}H_X)^{t_r}\cdot (N^{2r-2}H_X)^{t_{r-1}}\cdots (N^{2r-2j}H_X)^{t_{r-j}}\cdot H_X^{d-t_r-t_{r-1}- \cdots -t_{r-j}}>0.
\]
In particular, one has that $r\le d-1$.
\end{theorem}

As a consequence, we obtain an upper bound for the growth rate of $\deg_1(f^n)$, which is also optimal by \cite[Theorem~1.12]{Hu-GK-AV}.

\begin{corollary}
\label{cor:degree-growth-upper-bound}
Let $X$ be a normal projective variety of dimension $d$ over $\bk$.
Let $f$ be an automorphism of zero entropy of $X$.
Then we have that
\[
\deg_1(f^n) = O(n^{2(d-1)}), \text{ as } n\to \infty.
\]
Equivalently, the maximum size of Jordan blocks of $f^*|_{\N^1(X)_\bR}$ is at most $2d-1$.
\end{corollary}

The paper is organized as follows.
We first quote a useful lemma (see \cref{lemma:plov}) in \S\ref{subsec upper bound plov}, which gives a natural upper bound of $\plov$.
We then introduce weak numerical equivalence in \S\ref{subsec weak num equiv} and weak (semi)positivity in \S\ref{subsec weak positivity}.
Based on this, we construct explicit weakly positive classes and establish a strengthened version of \cref{thm:positivity} in \S\ref{sec upper bound k}, namely \cref{thm:full-version-positivity}.
In \S\ref{sec rank}, we define a series of matrices $A_{k,d,n}$ associated with restricted partitions and prove their full rank using representation theory of the Lie algebra $\mathfrak{sl}_2(\bC)$. This turns out to be the main ingredient in the proof of \cref{thm:plov}, which is presented in \S\ref{subsec proofs plov}.
Finally, in \S\ref{subsec examples}, we discuss the optimality of \cref{thm:plov} and explore possible extensions.

\section{Preliminaries}
\label{sec prel}

Throughout, unless otherwise stated, we work over an algebraically closed field $\bk$ of arbitrary characteristic, $X$ is a normal projective variety of dimension $d$ and $H_X$ is a fixed ample divisor on $X$, over $\bk$.
All {\it divisors} are assumed to be $\bR$-Cartier. 
We refer to Fulton \cite{Fulton98} for intersection theory on (possibly singular) algebraic varieties. We will mainly focus on intersections of divisors instead of arbitrary cycles.

\subsection{Upper bounds of \texorpdfstring{$\plov$}{plov} via intersection numbers}
\label{subsec upper bound plov}

Let $f$ be an automorphism of $X$.
Suppose that $f^*|_{\N^1(X)_\bR} = \id + N$ is unipotent, where $N$ is a nilpotent operator on $\N^1(X)_\bR$ satisfying that $N^k\neq 0$ and $N^{k+1}=0$ for some $k\in \bN$.
By the projection formula, we then have that
\[
\Vol(\Gamma_n) \coloneqq \Gamma_n \cdot \left(\sum_{i=1}^{n}\pr_i^*H_X\right)^d = \Delta_n(f,H_X)^d,
\]
where $\pr_i \colon X^n \to X$ is the projection to the $i$-th factor and
\begin{align*}
\Delta_n(f, H_X) &\coloneqq \sum_{m=0}^{n-1} (f^m)^*H_X = \sum_{m=0}^{n-1} \sum_{i=0}^{m} \binom{m}{i} N^i H_X \\
&= \sum_{i=0}^{n-1} \sum_{m=i}^{n-1} \binom{m}{i} N^i H_X = \sum_{i=0}^{k} \binom{n}{i+1} N^i H_X.
\end{align*}
Using this expression, the authors of \cite{LOZ} have obtained the following result on $\plov$.
See also \cite[Proof of Lemma~6.13]{Keeler00}.

\begin{lemma}[{cf.~\cite[Lemma~2.16]{LOZ}}]
\label{lemma:plov}
With notation as above, $\plov(f)$ is equal to the degree of the following polynomial (with indeterminate $n$)
\begin{align*}
\Delta_n(f, H_X)^d &= \Bigg( \sum_{i=0}^k \binom{n}{i+1} N^iH_X \Bigg)^d \\
&=\sum_{e_0+e_1+\cdots+e_k=d} \, \frac{d!}{e_0!e_1!\cdots e_k!} \, \prod_{i=0}^k \binom{n}{i+1}^{e_i} (N^{i}H_X)^{e_i}.
\end{align*}
In particular, one has that
\begin{align}\label{eq:plov-monomial}
\plov(f) \le d+\max\Bigg\{ \sum_{i=0}^k ie_i \, : \, \prod_{i=0}^k (N^{i}H_X)^{e_i}\neq 0, \, \sum_{i=0}^k e_i=d, \, e_i\in \bN \Bigg\}.
\end{align}
\end{lemma}

By the above lemma, we are particularly interested in the (non)vanishing of the intersection numbers $N^{i_1}H_X \cdots N^{i_d} H_X$ or $(N^{k}H_X)^{e_k} \cdots (N H_X)^{e_1} \cdot H_X^{e_0}$ with $\sum_{i=0}^k e_i=d$.
Since there is no need to distinguish the order of indices $i_1,\ldots,i_d$ in these intersection numbers, it is natural to consider restricted partitions (see \S\ref{subsec partition} for details).
On the other hand, we ask:

\begin{question}
\label{question:no-cancellation}
Is inequality \eqref{eq:plov-monomial} an equality?
Namely, if $(N^{k}H_X)^{e_k} \cdots (N H_X)^{e_1} \cdot H_X^{e_0} \neq 0$ for some $e_i\in\bN$ with $\sum_{i=0}^k e_i=d$, then do we have $\plov(f)\geq d+\sum_{i=0}^k ie_i$?
\end{question}

\begin{remark}
\label{rmk:conj-lower-bound}
Denote $k=2r$ and assume that $r>0$.
Then by \cref{thm:positivity}, there are positive integers $t_i\in\bZ_{>0}$ such that $\sum_{i=0}^r t_i = d$ and $(N^{2r}H_X)^{t_r}\cdots (N^2 H_X)^{t_1} \cdot H_X^{t_0}>0$.
If the above \cref{question:no-cancellation} has an affirmative answer, then we have that
\[
\plov(f) \ge d+\sum_{i=0}^r 2it_i \ge d+\sum_{i=0}^r 2i = d+r(r+1),
\]
which provides an optimal lower bound for $\plov(f)$ (cf. \cite[Theorem~5.1]{LOZ}).
Note that by \cite{Hu-GK-AV}, for any given $0<r\leq d-1$, there exists an automorphism $f$ of an abelian variety of dimension $d$ such that $\deg_1(f^n) \sim Cn^{2r}$ and $\plov(f)=d+r(r+1)$; see also \cref{ex optimal}.
\end{remark}

\subsection{Weak numerical equivalence}
\label{subsec weak num equiv}

For any $0\le i\le d$, denote by $\Z_i(X)$ the free abelian group of algebraic cycles of dimension $i$ on $X$.
An algebraic cycle $Z\in \Z_i(X)$ is {\it numerically trivial} and denoted by $Z\num 0$, if $Z\cdot P(\sE_I)=0$ for all weight $i$ homogeneous polynomials $P(\sE_I)$ in Chern classes of vector bundles on $X$ (see \cite[Definition~19.1]{Fulton98}).
Let $\N_i(X)$ denote the group of $i$-cycles on $X$ modulo numerical equivalence $\num$, which is a free abelian group of finite rank (see \cite[Example~19.1.4]{Fulton98}).
Denote $\N_i(X)_\bR\coloneqq \N_i(X) \otimes_\bZ \bR$.

As in \cite{FL17a,Dang20,Hu20a}, we let $\Z^i(X) \coloneqq \Hom_\bZ(\Z_i(X),\bZ)$ and
\[
\N^i(X) \coloneqq \Hom_\bZ(\N_i(X),\bZ)
\]
be the abstract dual groups of $\Z_i(X)$ and $\N_i(X)$, respectively.
Denote $\N^i(X)_\bR\coloneqq \N^i(X) \otimes_\bZ \bR$
the associated finite dimensional $\bR$-vector space and endow it with the standard Euclidean topology.
There is a natural map $\N^i(X)_\bR \lra \N_{d-i}(X)_\bR$
which may not be an isomorphism, in general (see \cite[Example~2.8]{FL17a}).
Nonetheless, we have that $D_1\cdot D_2\cdots D_i\in \N^i(X)_\bR$ for divisors $D_1,\ldots,D_i$ on $X$.

Note that $\N^1(X)_\bR$ is isomorphic to the real N\'eron--Severi space of divisors modulo numerical equivalence (see \cite[Example~2.1]{FL17a}).
However, cycle classes of higher dimension and their dual remain mysterious and are difficult to deal with.
It turns out to be helpful to consider the following notion (implicitly) introduced by Zhang \cite{Zhang09b}.

\begin{definition}[Weak numerical equivalence]
\label{def:weak-numerical-equiv}
A cycle class $\gamma\in \N_i(X)_\bR$ is called {\it weakly numerically trivial} and denoted by $\gamma\wnum 0$, if
\[
\gamma\cdot H_1\cdots H_{i} = 0
\]
for all ample (and hence for all) divisors $H_1,\ldots,H_{i}$ on $X$.

A dual cycle class $\alpha\in \N^i(X)_\bR$ is called {\it weakly numerically trivial} (and denoted by $\alpha \wnum 0$), if its image $\phi(\alpha)\wnum 0$ in $\N_{d-i}(X)_\bR$.
\end{definition}

By definition, weak numerical equivalence is indeed weaker than numerical equivalence.
Let $\W^i(X)$ denote the quotient group of $\N^i(X)$ modulo weak numerical equivalence, i.e.,
\[
\W^i(X)\coloneqq \N^i(X)/\!\wnum.
\]
In particular, $\W^i(X)_\bR\coloneqq \W^i(X) \otimes_\bZ \bR$ is a finite dimensional $\bR$-vector space, endowed with the natural quotient topology.
We have the following natural maps induced by the intersection of divisors:
\[
\N^1(X)_\bR\times \cdots\times \N^1(X)_\bR \lra \N^i(X)_\bR\longsurjmap \W^i(X)_\bR.
\]

\begin{lemma}
\label{lemma:K1-trivial}
$\N^1(X)_\bR= \W^1(X)_\bR$.
\end{lemma}
\begin{proof}
It follows from the higher dimensional Hodge index theorem; see \cite[Expos\'e~XIII, Corollaire~7.4]{SGA6} for $\bQ$-divisors and \cite[Proposition~2.9]{Hu20a} for $\bR$-divisors.
Precisely, let $D$ be a divisor such that $D\wnum 0$, i.e., $D\cdot D_1\cdots D_{d-1}=0$ for all divisors $D_1,\ldots,D_{d-1}$ on $X$, then
in particular $D\cdot H_1\cdots H_{d-1} = D^2\cdot H_1\cdots H_{d-2} = 0$ for all ample divisors $H_1,\ldots,H_{d-1}$ on $X$, which yields that $D\num 0$.
\end{proof}

For a divisor $D$, usually, $[D]$ represents its numerical class in $\N^1(X)_\bR$.
But by abuse of notation and the above lemma, we often simply write $D\in \N^1(X)_\bR$ or $\W^1(X)_\bR$.
Moreover, as we will be working over $\W^*(X)_\bR$, all intersections of divisor (classes) are considered to sit in $\W^*(X)_\bR$.

\begin{remark}\label{rem:WixW1}
It is worth mentioning that given any weakly numerically trivial classes $\alpha\in \N^i(X)_\bR$ and $\beta\in \N^j(X)_\bR$, we may not have $\alpha\cdot \beta\wnum 0$, since $\N^j(X)_\bR\to \W^j(X)_\bR$ may have a nontrivial kernel for $j\ge 2$.
So there is no natural intersection pairing $\W^i(X)\times \W^j(X)\to \W^{i+j}(X)$ for $i,j\ge 2$.
But $\W^i(X)\times \W^1(X)\to \W^{i+1}(X)$ is well defined by \cref{lemma:K1-trivial}.
\end{remark}

\subsection{Nef divisor product classes and weakly (semi)positive classes}
\label{subsec weak positivity}

In order to study the nonvanishing of intersection numbers of divisors, we introduce the following positivity notions which naturally arise when working in $\W^*(X)_\bR$.

\begin{definition}[Cone of nef divisor product classes]
\label{def:nef-product}
Let
\[
\NP^i(X)\coloneqq
\{D_1\cdots D_i\in \W^i(X)_\bR : D_1, \dots, D_i \in \Nef(X) \subset \N^1(X)_\bR\} 
\]
denote the cone of the intersections of $i$ nef divisor classes in $\W^i(X)_\bR$
and let $\NPbar^{i}(X)$ denote its closure in $\W^i(X)_\bR$. For convention, we formally set $\NP^0(X)=\NPbar^{0}(X):=\mathbf{R}_{\geq 0}\cdot [X]\subset \W^0(X)_\bR$.
\end{definition}

Note that $\NPbar^{i}(X)$ is not necessarily convex.

\begin{definition}[Weakly (semi)positive classes]
\label{def:weakly-positive}
A class $\alpha\in \W^i(X)_\bR$ is called {\it weakly positive} (resp. {\it weakly semipositive}), if $\alpha\cdot H_1\cdots H_{d-i}>0$ (resp. $\ge 0$) for any ample divisors $H_1, \dots, H_{d-i}$ on $X$.
It is denoted by $\alpha \wg 0$ (resp. $\alpha \wge 0$).
Here we recall that the intersection of a class in $\W^i(X)_\bR$ with ample divisors is well-defined by \cref{rem:WixW1}.
\end{definition}
For example, $\alpha\wge 0$ if $\alpha\in \NPbar^i(X)$ and we will often use this fact. 
\begin{lemma}
\label{lemma:alpha>0}
Let $\alpha\in \W^i(X)_\bR$ with $\alpha\wge 0$ and $\alpha\not \wnum 0$.
Then $\alpha\wg 0$.
\end{lemma}
\begin{proof}
Since $\alpha\not \wnum 0$, by \cref{def:weak-numerical-equiv}, there are ample divisors $H'_1, H'_2, \dots, H'_{d-i}$ on $X$ such that $\alpha\cdot H'_1\cdots H'_{d-i}$ is nonzero, which is necessarily positive because $\alpha\wge 0$.
For any ample divisors $H_1, H_2, \dots, H_{d-i}$ on $X$, there is a positive constant $\epsilon\ll 1$ such that $H_j-\epsilon H'_j$ is ample for each $1\le j\le d-i$.
Again by $\alpha\wge 0$, we have
\begin{align*} 
\alpha\cdot H_1\cdots H_{d-i} \geq \epsilon^{d-i}\alpha\cdot H'_1\cdots H'_{d-i} >0.
\end{align*}
This proves that $\alpha\wg 0$.
\end{proof}

The following lemma is essentially reminiscent of \cite[Corollaire~3.5]{DS04} (see also \cite[Lemma~2.14]{Hu20a}).

\begin{lemma}
\label{lemma:HR bilinear}
Fix an integer $0\leq i\leq d-2$.
Let $D, D'$ be divisors on $X$. 
Let $\alpha\in \NPbar^{i}(X)$ such that
$0\not \wnum \alpha \cdot D\in \NPbar^{i+1}(X)$ and $\alpha \cdot D^2\wge 0$.
\begin{enumerate}[label=\emph{(\arabic*)}, ref=(\arabic*)]
\item \label{lemma:HR bilinear-1} If $\alpha \cdot D\cdot D'\cdot H_1\cdots H_{d-i-2}= 0$ for some ample divisors $ H_1, \dots, H_{d-i-2}$ on $X$, then \[\alpha\cdot D'^2\cdot H_1\cdots H_{d-i-2}\leq 0.\] 
Moreover, if the equality holds, then there exists a unique constant $c\in\bR$, depending on $H_1, \dots, H_{d-i-2}$, such that 
\[
\alpha\cdot (D'+cD)\cdot H_1\cdots H_{d-i-2} \wnum 0.
\]

\item \label{lemma:HR bilinear-2} If $\alpha \cdot D\cdot D'\wnum 0$, then $-\alpha\cdot D'^2\wge 0$.

\item \label{lemma:HR bilinear-3} If $\alpha \cdot D\cdot D'\wnum \alpha\cdot D'^2\wnum 0$, then there exists a unique constant $c\in\bR$ such that $\alpha\cdot (D'+cD) \wnum 0$.
\end{enumerate}
\end{lemma}

\begin{proof}
(1) For the given ample divisors $H_1, \dots, H_{d-i-2}$ on $X$, we define a symmetric bilinear form $q$ on $\N^1(X)_\bR$ as follows:
\begin{equation*}
q(D_1,D_2) \coloneqq D_1\cdot D_2 \cdot \alpha \cdot H_1 \cdots H_{d-i-2}, \textrm{ for any } D_1,D_2\in \N^1(X)_\bR.
\end{equation*}
Fix an auxiliary ample divisor $H$ on $X$.
Note that $\alpha \cdot D\not \wnum 0$ (and hence $\alpha \not \wnum 0$).
So by \cref{lemma:alpha>0}, we have
$q(D, H)>0$ and $q(H, H)>0$.
The primitive subspace $P_H \subset \N^1(X)_\bR$ associated with $q$ and $H$ is defined by
\[
P_H \coloneqq \{D_2\in \N^1(X)_\bR : q(H,D_2) = 0\},
\]
which is a hyperplane in $\N^1(X)_\bR$.

We claim that $q$ is negative semidefinite on $P_H$. In fact, if $\alpha$ is a product of ample divisors on $X$, then the bilinear form $q$ is negative definite on $P_H$ by the higher dimensional Hodge index theorem (see \cite[Expos\'e~XIII, Corollaire~7.4]{SGA6} or \cite[Proposition~2.9]{Hu20a}).
As $\alpha\in \NPbar^{i}(X)$, by continuity, $q$ is negative semidefinite on $P_H$.

Since $q(D, H)>0$, $D$ is not in the hyperplane $P_H$.
Hence there is a unique constant $c\in \bR$ such that $D'+cD\in P_H$,
precisely, $c=-q(H,D')/q(H,D)$.
So by the assumption that $q(D, D')=0$ and the negative semidefiniteness, one has
\begin{align}\label{eq:D+cD'<=0}
q(D', D')+c^2q(D, D)=q(D'+cD, D'+cD)\leq 0.
\end{align}
On the other hand, $q(D, D)\geq 0$ by $\alpha \cdot D^2\wge 0$. So $q(D', D')\leq 0$. This proves the first half part of Assertion~\ref{lemma:HR bilinear-1}.

We now suppose that $q(D', D')= 0$.
Then \cref{eq:D+cD'<=0} implies that $q(D'+cD, D'+cD)=0$.
To finish the proof, it suffices to show that $q(D'+cD, D_2)=0$ for any divisor $D_2$ on $X$.

Given any $D_2$, since $q(H, H)>0$, as above, we may take $e=-q(H, D_2)/q(H, H)$ such that $D_2+eH\in P_H$.
So by the Cauchy--Schwarz inequality for $q|_{P_H}$,
\[
q(D'+cD,D_2+eH)^2 \le q(D'+cD,D'+cD) \, q(D_2+eH,D_2+eH) = 0.
\]
This implies that $q(D'+cD, D_2)=q(D'+cD,D_2+eH)=0$ where the first equality follows from $D'+cD\in P_H$.

(2) It follows directly from Assertion~\ref{lemma:HR bilinear-1}.

(3) By assumption and Assertion~\ref{lemma:HR bilinear-1}, for any ample divisors $H_0, H_1, \dots, H_{d-i-2}$ on $X$, there is a unique constant $c\in\bR$, depending on $H_1, \dots, H_{d-i-2}$ but not on $H_0$, such that 
\begin{align}\label{qe:c1}
\alpha\cdot (D'+cD) \cdot H_0\cdot H_1\cdots H_{d-i-3}\cdot H_{d-i-2}=0. 
\end{align}
It suffices to show that $c$ is independent of $H_1, \dots, H_{d-i-2}$.
Let $H'_{d-i-2}$ be an arbitrary ample divisor on $X$. Applying Assertion~\ref{lemma:HR bilinear-1} to ample divisors $H_1, \dots, H_{d-i-3}, H'_{d-i-2}$, there is a unique constant $c'\in\bR$, depending on $H_1, \dots, H_{d-i-3}, H'_{d-i-2}$ but not on $H_0$, such that
\begin{align}\label{qe:c2}
\alpha\cdot (D'+c'D) \cdot H_0\cdot H_1\cdots H_{d-i-3}\cdot H'_{d-i-2}=0.\end{align}
Comparing \cref{qe:c1} with \cref{qe:c2}, we can apply the uniqueness in Assertion~\ref{lemma:HR bilinear-1} to ample divisors $H_0, H_1, \dots, H_{d-i-3}$ to get $c=c'$. This shows that $c$ is independent of the choice of $H_{d-i-2}$, and hence of the choice of all $H_j$. 
\end{proof}

\section{An upper bound for polynomial growth rate of degree sequences}
\label{sec upper bound k}

In this section we study the polynomial growth rate of degree sequences and prove \cref{thm:positivity} and \cref{cor:degree-growth-upper-bound}.

\subsection{Construction of weakly positive nef divisor products}
\label{subsec construction}

The main result in this subsection is the following \cref{thm:full-version-positivity}, which strengthens \cref{thm:positivity} and particularly asserts that the product $N^{2r}H_X\cdot N^{2r-2}H_X\cdots N^2H_X\cdot H_X \not\wnum 0$ in $\W^{r+1}(X)_\bR$.

We will always keep the following assumption, which is exactly the same as in \cref{thm:positivity}.

\begin{assumption}\label{assum:XfN}
Let $X$ be a normal projective variety of dimension $d$ and $H_X$ a fixed ample divisor on $X$, over $\bk$.
Let $f$ be an automorphism of $X$ such that $f^*|_{\N^1(X)_\bR}=\id+N$ is unipotent with $N^{k}\neq 0$ and $N^{k+1}=0$, where $k=2r$ is a positive even integer.
\end{assumption}

\begin{theorem}
\label{thm:full-version-positivity}
Keep the notation in \cref{assum:XfN}. Then there are positive integers $t_r, t_{r-1}, \dots, t_1$ with $t_r+t_{r-1}+\cdots+t_1<d$ satisfying the following property.

Denote $s_{0}\coloneqq 0$, $M_0\coloneqq [X]$,
and for $1\leq j\leq r$,
$s_j\coloneqq t_r+t_{r-1}+\cdots+t_{r-j+1}$,
\[
M_{j}\coloneqq (N^{2r}H_X)^{t_r}\cdots (N^{2r-2j+2}H_X)^{t_{r-j+1}}\in \W^{s_j}(X)_\bR.
\]
Then for each $0\leq j\leq r-1$, the following assertions hold: 
\begin{enumerate}[label=\emph{(\arabic*)}, ref=(\arabic*)]
\item \label{condition NPbar}
$M_j\cdot (N^{2r-2j}H_X)^{i}\in \NPbar^{s_{j}+i}(X)$ for any $0\leq i\leq t_{r-j}$;

\item \label{condition f-invariant}
$f^*(M_j\cdot (N^{2r-2j}H_X)^{i})\wnum M_j \cdot (N^{2r-2j}H_X)^{i}$ for any $0\leq i\leq t_{r-j}$;

\item \label{condition NNN ample}
$M_j\cdot (N^{2r-2j}H_X)^{i}\wg 0$ for any $0\leq i\leq t_{r-j}$; 

\item \label{condition N2r-2j-1 =0}
$M_{j+1}\cdot N^{s}H_X\wnum 0$ for all $s\geq 2r-2j-1$, where $M_{j+1}=M_j\cdot (N^{2r-2j}H_X)^{t_{r-j}}$;

\item \label{condition N2r-2j-1^2 <0}

$
-M_j\cdot (N^{2r-2j}H_X)^{t_{r-j}-1}\cdot (N^{2r-2j-1}H_X)^2 \wg 0;
$
 
\item \label{condition N2r-2j-2 >0}
$M_{j+1}\cdot N^{2r-2j-2}H_X\wg 0$.
\end{enumerate}
\end{theorem}

\begin{remark}
\label{rmk:idea-positivity}
\cref{thm:positivity} is a simplified version of the above \cref{thm:full-version-positivity}.
The idea of its proof is a mixture of the proofs of \cref{thm:plov} and \cite[Theorem~1.1]{DLOZ22}.
Note that \cref{thm:plov} follows essentially from the vanishing of intersection numbers $N^{i_1}H_X \cdots N^{i_d} H_X$ with $\sum_{j=1}^d i_j > dr$ by \cref{lemma:plov}.
So it is natural to consider the nonvanishing of these intersection numbers (see \cref{question:no-cancellation}).
Towards this, we prove the weak positivity of certain products $(N^{2r}H_X)^{t_r}\cdot (N^{2r-2}H_X)^{t_{r-1}}\cdots (N^{2r-2j}H_X)^{t_{r-j}}$ (see \cref{def:weakly-positive}).
It turns out that these products explicitly form a quasi-nef sequence as introduced in \cite{Zhang09b}; however, we now do not need any knowledge of dynamical filtrations developed in \cite[\S3]{DLOZ22}.
\end{remark}

We shall repeatedly use the following two lemmas in the proof of \cref{thm:full-version-positivity}. They describe the polynomial behavior of intersection numbers of divisors under certain pullbacks by $f^*$.

\begin{lemma}
\label{lemma:polynomial}
Keep the notation in \cref{assum:XfN}. Let $\beta\in \W^i(X)_\bR$ with $\beta\wg 0$.
Fix an integer $l$ with $0\le l\le d-i$ and ample divisors $H_1, \dots, H_{d-i-l}$ on $X$.
Consider the intersection number
\[
P(m)=\beta\cdot ((f^m)^*H_X)^l\cdot H_1\cdots H_{d-i-l}, \quad m\in\bZ.
\]
Then the following assertions hold.
\begin{enumerate}[label=\emph{(\arabic*)}, ref=(\arabic*)]
\item \label{assert 1 of lemma:polynomial} $P(m)$ is a positive polynomial in $m$.

\item \label{assert 2 of lemma:polynomial}
The leading coefficient of $P(m)$ is positive.

\item \label{assert 3 of lemma:polynomial}
The degree of $P(m)$ is even. 
 
\end{enumerate}
\end{lemma}

\begin{proof}
Note that $(f^m)^*H_X = (\id+N)^mH_X \num \sum_{j=0}^{2r} \binom{m}{j}N^jH_X$.
So
\begin{align}
\label{eq:poly coeff}
\begin{split}
P(m) &= \beta\cdot ((f^m)^*H_X)^l\cdot H_1\cdots H_{d-i-l} \\
&= \sum_{0\leq j_1, \dots, j_l\leq 2r}\binom{m}{j_1}\cdots \binom{m}{j_l} \ \beta\cdot N^{j_1}H_X\cdots N^{j_l}H_X \cdot H_1\cdots H_{d-i-l}
\end{split}
\end{align}
is a polynomial in $m$.
Since $f$ is an automorphism and $\beta\wg 0$ by assumption, $P(m)$ is always positive.
So the degree of $P(m)$ is even and the leading coefficient of $P(m)$ is positive. 
\end{proof}

We remark that the degree of $P(m)$ is independent of the choice of $H_X, H_1, \dots, H_{d-i-l}$. The proof is simply by comparing intersection numbers as in the proof of \cref{lemma:alpha>0}, but we do not need this fact in the paper, so we omit the details.

In practice, we will frequently use the expanded \cref{eq:poly coeff} to compute the degree and coefficients of $P(m)$ by the following more precise lemma in different situations.
Though its proof is straightforward by a direct computation, we include it for the reader's convenience.

\begin{lemma}
\label{lemma:polynomial2}
Keep the assumption in \cref{lemma:polynomial}. 
Suppose further that $l>0$ and there exists an integer $j$ with $0\leq j\leq r-1$ such that $\beta\cdot N^{s}H_X\wnum 0$ for all $s\geq 2r-2j+1$.
Then the following assertions hold.
\begin{enumerate}[label=\emph{(\arabic*)}, ref=(\arabic*)]
\item \label{assert 1 of coeff lemma} $\deg P(m)\leq (2r-2j)l$.
\item \label{assert 2 of coeff lemma} The coefficient of the term $m^{(2r-2j)l}$ in $P(m)$ is a positive multiple of
\begin{align}\label{eq:leading coeff of P}
 \beta\cdot (N^{2r-2j}H_X)^{l}\cdot H_1\cdots H_{d-i-l},
\end{align}
which is nonnegative.

\item \label{assert 3 of coeff lemma} $\deg P(m)= (2r-2j)l$ if and only if \eqref{eq:leading coeff of P} is nonzero.

\item \label{assert 4 of coeff lemma} If $\deg P(m)<(2r-2j)l$, then $\deg P(m)\leq (2r-2j)l-2$, and 
\begin{align}\label{eq:leading coeff of P-1}
 \beta\cdot (N^{2r-2j}H_X)^{l-1}\cdot N^{2r-2j-1}H_X\cdot H_1\cdots H_{d-i-l}=0.
\end{align}
Furthermore, the coefficient of the term $m^{(2r-2j)l-2}$ in $P(m)$ is a positive multiple of \begin{align}\label{eq:leading coeff of P-2}
\begin{split}
{} & \beta\cdot (N^{2r-2j}H_X)^{l-1}\cdot N^{2r-2j-2}H_X\cdot H_1\cdots H_{d-i-l} \\
+{} \ & \frac{(l-1)(r-j)}{2r-2j-1}\ \beta\cdot (N^{2r-2j}H_X)^{l-2}\cdot (N^{2r-2j-1}H_X)^2\cdot H_1\cdots H_{d-i-l},
\end{split}
\end{align}
which is nonnegative. Here the second term in \eqref{eq:leading coeff of P-2} does not appear if $l=1$. 
\end{enumerate}
\end{lemma}

\begin{proof}
By the assumption that $\beta\cdot N^{s}H_X\wnum 0$ for all $s\geq 2r-2j+1$, the sum in \cref{eq:poly coeff} runs over ${0\leq j_1, \dots, j_l\leq 2r-2j}$.
Hence we get Assertions~\ref{assert 1 of coeff lemma}--\ref{assert 3 of coeff lemma}. Here \eqref{eq:leading coeff of P} is nonnegative as $P(m)$ is a positive polynomial by \cref{lemma:polynomial}~\ref{assert 1 of lemma:polynomial}. 

For Assertion~\ref{assert 4 of coeff lemma}, if $\deg P(m)<(2r-2j)l$, then $\deg P(m)\leq (2r-2j)l-2$ as $\deg P(m)$ is even by \cref{lemma:polynomial}~\ref{assert 3 of lemma:polynomial}.
Then by considering the coefficient of the term $m^{(2r-2j)l-1}$ in $P(m)$, we get \cref{eq:leading coeff of P-1}.
Similarly, by considering the coefficient of $m^{(2r-2j)l-2}$, we see that \eqref{eq:leading coeff of P-2} is nonnegative as $P(m)$ is a positive polynomial. 
\end{proof}

We shall prove \cref{thm:full-version-positivity} by induction on the index $j$. To make the proof more accessible, we begin with the following proposition which is exactly the base case $j=0$.

\begin{proposition}
\label{prop:zero-case}
Keep the notation in \cref{assum:XfN}. Then there exists a positive integer $t_r<d$ such that the following assertions hold:
\begin{enumerate}[label=\emph{(\arabic*)}, ref=(\arabic*)]

\item \label{prop:zero-case-1}
$(N^{2r}H_X)^{i} \in \NPbar^{i}(X)$ for any $0\leq i\leq t_{r}$, where we formally set $(N^{2r}H_X)^{0}=[X]$; 

\item \label{prop:zero-case-2}
$f^*((N^{2r}H_X)^{i})\wnum (N^{2r}H_X)^{i}$ for any $0\leq i\leq t_{r}$;

\item \label{prop:zero-case-3}
$(N^{2r}H_X)^{i}\wg 0$ for any $0\leq i\leq t_{r}$;

\item \label{prop:zero-case-4}
$(N^{2r}H_X)^{t_{r}}\cdot N^{s}H_X\wnum 0$ for all $s\geq 2r-1$; 

\item \label{prop:zero-case-5}

$
-(N^{2r}H_X)^{t_{r}-1}\cdot (N^{2r-1}H_X)^2 \wg 0; 
$

\item \label{prop:zero-case-6}
$(N^{2r}H_X)^{t_{r}}\cdot N^{2r-2}H_X\wg 0$.
\end{enumerate}
In particular, $M_1\coloneqq (N^{2r}H_X)^{t_r}\wg 0$ and $M_1\cdot N^{2r-2}H_X\wg 0$.
\end{proposition}
\begin{proof}
Note first that $N^{2r}H_X\not \equiv 0$ by \cite[Lemma~4.4]{Keeler00}.
So we define $t_r$ by 
\begin{equation}
\label{eq:def-tr}
t_r \coloneqq \max\{t\in \bZ_{>0} : (N^{2r}H_X)^{t}\not \wnum 0\}.
\end{equation}
In other words,
$(N^{2r}H_X)^{t_r}\not \wnum 0$, but $(N^{2r}H_X)^{t_r+1} \wnum 0$.
Clearly, $t_r\leq d$.

We claim that $t_r<d$.
Otherwise, $(N^{2r}H_X)^{d}\neq 0$
and it is the coefficient of the term $m^{2rd}$ in the polynomial $((f^m)^*H_X)^d$ using \cref{eq:poly coeff}. On the other hand, $((f^m)^*H_X)^d=H_X^d$ is constant by the projection formula, which is a contradiction.
In the following, we just check all assertions.

\smallskip

(1) Recall that $(f^m)^*H_X\num \sum_{i=0}^{2r}\binom{m}{i}N^iH_X$.
It follows that
\[
N^{2r}H_X\num \lim_{m\to \infty}\left(\frac{(2r)!}{m^{2r}}(f^m)^*H_X\right) \in \Nef(X)
\]
is a nef divisor class.
So $(N^{2r}H_X)^{i}$ is in $\NP^{i}(X)$ and hence in $\NPbar^{i}(X)$ for any $0\leq i\leq t_{r}$.

(2) Since $f^*(N^{2r}H_X)\equiv N^{2r}H_X+ N^{2r+1}H_X\equiv N^{2r}H_X$, we get that $f^*((N^{2r}H_X)^{i})\wnum (N^{2r}H_X)^{i}$ for any $0\leq i\leq t_{r}$.

\smallskip

(3) By the definition of $t_r$, i.e., \cref{eq:def-tr}, we have $(N^{2r}H_X)^{i}\not\wnum 0$ for any $0\leq i\leq t_{r}$.
Assertion~\ref{prop:zero-case-3} thus follows from Assertion~\ref{prop:zero-case-1} and \cref{lemma:alpha>0}. 

\smallskip

(4) First, by \cref{assum:XfN} and the construction of $t_r$, $(N^{2r}H_X)^{t_r}\cdot N^{s}H_X\wnum 0$ for all $s\ge 2r$.
It suffices to show that $(N^{2r}H_X)^{t_r}\cdot N^{2r-1}H_X\wnum 0$.
Let $H_1, \dots, H_{d-t_r-1}$ be arbitrary ample divisors on $X$.
Consider the following polynomial 
\begin{align*}
P(m) \coloneqq ((f^m)^*H_X)^{t_r+1}\cdot H_1\cdots H_{d-t_r-1}.
\end{align*}
Namely, we will apply \cref{lemma:polynomial,lemma:polynomial2} for $\beta=[X]$, $i=j=0$, and $l=t_{r}+1$.
Again, by the definition of $t_r$, $(N^{2r}H_X)^{t_r+1} \wnum 0$.
So by \cref{lemma:polynomial2} \ref{assert 1 of coeff lemma}--\ref{assert 3 of coeff lemma}, $\deg P(m)<2r(t_r+1)$.
Then according to \cref{eq:leading coeff of P-1} in \cref{lemma:polynomial2}~\ref{assert 4 of coeff lemma}, we get that
\begin{align*}
(N^{2r}H_X)^{t_r}\cdot N^{2r-1}H_X\cdot H_1\cdots H_{d-t_r-1} = 0.
\end{align*}
Since $H_1, \dots, H_{d-t_r-1}$ are arbitrary, $(N^{2r}H_X)^{t_r}\cdot N^{2r-1}H_X\wnum 0$.
So Assertion~\ref{prop:zero-case-4} is proved.

\smallskip

(5) Applying \cref{lemma:HR bilinear}~\ref{lemma:HR bilinear-2} to $\alpha=(N^{2r}H_X)^{t_r-1}$, $D=N^{2r}H_X$, and $D'=N^{2r-1}H_X$ by Assertions \ref{prop:zero-case-1} and \ref{prop:zero-case-4}, we obtain that
\[
-(N^{2r}H_X)^{t_r-1}\cdot (N^{2r-1}H_X)^2\wge 0.\]
By \cref{lemma:alpha>0}, it suffices to show that $
(N^{2r}H_X)^{t_r-1}\cdot (N^{2r-1}H_X)^2\not \wnum 0.$
Suppose to the contrary that $
(N^{2r}H_X)^{t_r-1}\cdot (N^{2r-1}H_X)^2 \wnum 0$.
Then by \cref{lemma:HR bilinear}~\ref{lemma:HR bilinear-3}, there exists a constant $c\in\bR$ such that 
\[
(N^{2r}H_X)^{t_r-1}\cdot (N^{2r-1}H_X+cN^{2r}H_X) \wnum 0.
\]
Applying $f^*$ to the above equality and noting that $f^*((N^{2r}H_X)^{t_r-1})\wnum (N^{2r}H_X)^{t_r-1}$ by Assertion~\ref{prop:zero-case-2}, we have that
\[
(N^{2r}H_X)^{t_r-1}\cdot (N^{2r-1}H_X+(1+c)N^{2r}H_X) \wnum 0.
\]
But this implies that $(N^{2r}H_X)^{t_r} \wnum 0$, a contradiction.
So Assertion~\ref{prop:zero-case-5} is proved.

\smallskip

(6) We go back to the proof of Assertion~\ref{prop:zero-case-4}.
Recall that we have shown there that $\deg P(m)<2r(t_r+1)$.
Hence by \cref{lemma:polynomial2}~\ref{assert 4 of coeff lemma}, the coefficient of the term $m^{2r(t_r+1)-2}$ in $P(m)$, which is a positive multiple of
\begin{align*}
{} & (N^{2r}H_X)^{t_r}\cdot N^{2r-2}H_X\cdot H_1\cdots H_{d-t_r-1} \\
+ \ & \frac{rt_r}{2r-1} \, (N^{2r}H_X)^{t_r-1}\cdot (N^{2r-1}H_X)^2\cdot H_1\cdots H_{d-t_r-1},
\end{align*}
is nonnegative.
Note that by Assertion~\ref{prop:zero-case-5}, one has
\[
(N^{2r}H_X)^{t_r-1}\cdot (N^{2r-1}H_X)^2\cdot H_1\cdots H_{d-t_r-1} <0.
\]
It follows that $(N^{2r}H_X)^{t_r}\cdot N^{2r-2}H_X\cdot H_1\cdots H_{d-t_r-1}>0$.
Since $H_1, \dots, H_{d-t_r-1}$ are arbitrary, we get that $(N^{2r}H_X)^{t_r}\cdot N^{2r-2}H_X\wg 0$, and hence complete the proof of \cref{prop:zero-case}.
\end{proof}

\begin{proof}[Proof of \cref{thm:full-version-positivity}]
We prove the theorem by induction on $j$.
Namely, we will inductively construct the sequence of positive integers $t_r, \dots, t_{r-j}, \dots, t_1$ with $M_{j+1}\coloneqq M_j\cdot (N^{2r-2j}H_X)^{t_{r-j}}$ satisfying all assertions. 
The case when $j=0$ has been proved in \cref{prop:zero-case}.

Suppose that we have constructed $t_r, \dots, t_{r-(j-1)}$ for some $0\leq j-1\leq r-2$.
Denote $s_j\coloneqq t_r+t_{r-1}+\cdots+t_{r-(j-1)}$.
We will construct $t_{r-j}$ satisfying all assertions.
So from the inductive hypotheses for $j-1$, we particularly have the following:
\begin{enumerate}[label=(IH\arabic*), ref=IH\arabic*] 
\item \label{IH NPbar}
$M_j\in \NPbar^{s_{j}}(X)$;

\item \label{IH f-invariant}
$f^*(M_j)\wnum M_j$;

\item \label{IH NNN ample} 
$M_j\wg 0$; 

\item \label{IH N2r-2j-1 =0}
$M_j\cdot N^{s}H_X\wnum 0$ for all $s\geq 2r-2j+1$;

\setcounter{enumi}{5}
\item \label{IH N2r-2j-2 >0} 
$M_j\cdot N^{2r-2j}H_X\wg 0$.
\end{enumerate}

By \eqref{IH N2r-2j-2 >0}, 
we define $t_{r-j}$ as follows:
\begin{equation}
\label{eq:def-tr-j}
t_{r-j} \coloneqq \max\{t\in \mathbf{Z}_{>0} : M_j\cdot (N^{2r-2j}H_X)^{t}\not \wnum 0\}.
\end{equation}
In particular, $s_{j+1}=s_j+t_{r-j}\le d$.

\smallskip

(1) We have that for any $0\leq i\leq t_{r-j}$,
\begin{equation}
\label{eq:Mj limit}
M_j\cdot (N^{2r-2j}H_X)^{i}
\wnum \lim_{m\to \infty} \left( M_j\cdot \frac{((2r-2j)!)^i}{ m^{i(2r-2j)}}((f^m)^*H_X)^i\right) \in \NPbar^{s_{j}+i}(X).
\end{equation}
Here the equality follows from \eqref{IH N2r-2j-1 =0} and the inclusion follows from \eqref{IH NPbar}.

 We then claim that $s_{j+1}=s_j+t_{r-j}<d$ as required in \cref{thm:full-version-positivity}. In fact, if $s_{j+1}=d$, then by the projection formula and \eqref{IH f-invariant}, 
\[M_j\cdot ((f^m)^*H_X)^{t_{r-j}}=(f^{-m})^{*}M_j\cdot H_X^{t_{r-j}}=M_j\cdot H_X^{t_{r-j}}\] is a constant.
Hence $M_j\cdot (N^{2r-2j}H_X)^{t_{r-j}}=0$ because the limit on the right hand side of \cref{eq:Mj limit} is $0$ when $i=t_{r-j}$.
This contradicts to the definition of $t_{r-j}$, i.e., \cref{eq:def-tr-j}.

\smallskip

(2) Assertion \ref{condition f-invariant} follows directly by
\begin{align*}
 {}& f^*(M_j\cdot (N^{2r-2j}H_X)^{i})\wnum M_j\cdot f^*((N^{2r-2j}H_X)^{i})\\
 \wnum {}& M_j\cdot (N^{2r-2j}H_X+N^{2r-2j+1}H_X)^{i} \wnum M_j \cdot (N^{2r-2j}H_X)^{i},
\end{align*} for any $0\leq i\leq t_{r-j}$.
Here the first equality follows from \eqref{IH f-invariant} and the last equality follows from \eqref{IH N2r-2j-1 =0}.

\smallskip

(3) By the definition of $t_{r-j}$, i.e., \cref{eq:def-tr-j}, we have $M_j\cdot (N^{2r-2j}H_X)^{i}\not\wnum 0$ for any $0\leq i\leq t_{r-j}$.
Assertion \ref{condition NNN ample} thus follows from Assertion \ref{condition NPbar} and \cref{lemma:alpha>0}.

\smallskip
 
(4)
To prove Assertion \ref{condition N2r-2j-1 =0}, by \eqref{IH N2r-2j-1 =0} and the construction of $t_{r-j}$, it suffices to show that 
\begin{align}
\label{eq:proof j>0 2r-1=0}
 M_j\cdot (N^{2r-2j}H_X)^{t_{r-j}}\cdot N^{2r-2j-1}H_X\wnum 0.
\end{align}
Let $H_1, \dots, H_{d-s_{j+1}-1}$ be arbitrary ample divisors on $X$.
We then consider the following polynomial 
\[
Q(m) \coloneqq M_j\cdot ((f^m)^*H_X)^{t_{r-j}+1}\cdot H_1\cdots H_{d-s_{j+1}-1}.
\]
By \eqref{IH NNN ample} and \eqref{IH N2r-2j-1 =0},
we can apply \cref{lemma:polynomial,lemma:polynomial2} for $\beta= M_j$ and $l=t_{r-j}+1$.
As $M_j\cdot (N^{2r-2j}H_X)^{t_{r-j}+1} \wnum 0$, by \cref{lemma:polynomial2} \ref{assert 1 of coeff lemma}--\ref{assert 3 of coeff lemma}, 
\[\deg Q(m)<(2r-2j)(t_{r-j}+1).\]
Then according to \cref{eq:leading coeff of P-1} in \cref{lemma:polynomial2}~\ref{assert 4 of coeff lemma}, we get
\begin{align*}
 M_j\cdot (N^{2r-2j}H_X)^{t_{r-j}}\cdot N^{2r-2j-1}H_X \cdot H_1\cdots H_{d-s_{j+1}-1} = 0. 
\end{align*}
Since $H_1, \dots, H_{d-s_{j+1}-1}$ are arbitrary, \cref{eq:proof j>0 2r-1=0} and hence Assertion \ref{condition N2r-2j-1 =0} follow.

\smallskip

(5) By Assertions \ref{condition NPbar} and \ref{condition N2r-2j-1 =0}, we can apply \cref{lemma:HR bilinear}~\ref{lemma:HR bilinear-2} to $\alpha=M_j\cdot (N^{2r-2j}H_X)^{t_{r-j}-1}$, $D=N^{2r-2j}H_X$, and $D'=N^{2r-2j-1}H_X$.
We thus obtain that
\[
-M_j\cdot (N^{2r-2j}H_X)^{t_{r-j}-1}\cdot (N^{2r-2j-1}H_X)^2 \wge 0.
\]
By \cref{lemma:alpha>0}, it suffices to show that $M_j\cdot (N^{2r-2j}H_X)^{t_{r-j}-1}\cdot (N^{2r-2j-1}H_X)^2\not\wnum 0$.
Suppose that this is not the case.
Then by \cref{lemma:HR bilinear}~\ref{lemma:HR bilinear-3}, there exists a constant $c\in\bR$ such that
\[
M_j\cdot (N^{2r-2j}H_X)^{t_{r-j}-1}\cdot (N^{2r-2j-1}H_X+cN^{2r-2j}H_X) \wnum 0.
\] 
But applying $f^*$ to this equality, we have
\begin{align*}
0{}& \wnum f^*(M_j\cdot (N^{2r-2j}H_X)^{t_{r-j}-1}\cdot (N^{2r-2j-1}H_X+cN^{2r-2j}H_X))\\
{}&\wnum M_j\cdot (N^{2r-2j}H_X)^{t_{r-j}-1}\cdot f^*(N^{2r-2j-1}H_X+cN^{2r-2j}H_X)\\
{}&\wnum 
M_j\cdot (N^{2r-2j}H_X)^{t_{r-j}-1}\cdot (N^{2r-2j-1}H_X+(c+1)N^{2r-2j}H_X).
\end{align*}
Here the second equality follows from Assertion \ref{condition f-invariant} and the last equality follows from \eqref{IH N2r-2j-1 =0}.
By taking subtraction of the above two equations, we get
$
M_j\cdot (N^{2r-2j}H_X)^{t_{r-j}}\wnum 0,$
which contradicts to the definition of $t_{r-j}$, i.e., \cref{eq:def-tr-j}.

\smallskip

(6) We go back to the proof of Assertion \ref{condition N2r-2j-1 =0}.
Recall that $\deg Q(m)<(2r-2j)(t_{r-j}+1)$.
So by \cref{lemma:polynomial2} \ref{assert 4 of coeff lemma}, the coefficient of the term $m^{(2r-2j)(t_{r-j}+1)-2}$ in $Q(m)$, which is a positive multiple of
\begin{align*}
{} & M_j\cdot (N^{2r-2j}H_X)^{t_{r-j}}\cdot N^{2r-2j-2}H_X\cdot H_1\cdots H_{d-s_{j+1}-1} \\
+{} \ & \frac{t_{r-j}(r-j)}{2r-2j-1}\ M_j\cdot (N^{2r-2j}H_X)^{t_{r-j}-1}\cdot (N^{2r-2j-1}H_X)^2\cdot H_1\cdots H_{d-s_{j+1}-1},
\end{align*}
is nonnegative.
Note that by Assertion \ref{condition N2r-2j-1^2 <0}, we have
\[
M_j\cdot (N^{2r-2j}H_X)^{t_{r-j}-1}\cdot (N^{2r-2j-1}H_X)^2\cdot H_1\cdots H_{d-s_{j+1}-1}< 0.
\]
It follows that
\[
M_j\cdot (N^{2r-2j}H_X)^{t_{r-j}}\cdot N^{2r-2j-2}H_X\cdot H_1\cdots H_{d-s_{j+1}-1}>0.
\]
Hence $M_j\cdot (N^{2r-2j}H_X)^{t_{r-j}}\cdot N^{2r-2j-2}H_X\wg 0$ as $H_1,\ldots, H_{d-s_{j+1}-1}$ are arbitrary.

\smallskip

We thus complete the proof of \cref{thm:full-version-positivity} by induction.
\end{proof}

\subsection{Proofs of Theorem \ref{thm:positivity} and Corollary \ref{cor:degree-growth-upper-bound}}

\begin{proof}[Proof of \cref{thm:positivity}]
By \cref{thm:full-version-positivity} \ref{condition NNN ample}, there are positive integers $t_r,\ldots,t_1$ with $\sum_{i=1}^{r} t_i <d$ (in particular, $r<d$), such that for any $0\le j\le r-1$,
\[
M_{j+1}=(N^{2r}H_X)^{t_r}\cdot (N^{2r-2}H_X)^{t_{r-1}}\cdots (N^{2r-2j}H_X)^{t_{r-j}} \wg 0.
\]
Then \cref{thm:positivity} follows by \cref{def:weakly-positive}.
\end{proof}

Though the following lemma should be well known, we include it here for the sake of completeness.

\begin{lemma}
\label{lemma:deg-growth-invariant-under-iteration}
Let $X$ be a normal projective variety of dimension $d$ and $H_X$ an ample divisor on $X$, over $\bk$.
Let $f$ be an automorphism of zero entropy of $X$.
Then for any fixed $m\in \bZ_{>0}$, $\deg_i(f^n) = O(\deg_i(f^{mn}))$, as $n\to \infty$.
\end{lemma}
\begin{proof}
It is well known that $\deg_i(f^n)$ and $\|(f^n)^*|_{\N^i(X)_\bR}\|$ have the same growth rate (see, e.g., \cite{Truong20,Dang20}).
By definition, $f^*|_{\N^1(X)_\bR}$ is quasi-unipotent, so are all $f^*|_{\N^i(X)_\bR}$ by the log-concavity property of dynamical degrees (see \cite[Lemma~5.7]{Truong20}).
We just need to note that the maximum size of Jordan blocks of $f^*|_{\N^i(X)_\bR}$ and $(f^m)^*|_{\N^i(X)_\bR}$ coincide.
\end{proof}

\begin{proof}[Proof of \cref{cor:degree-growth-upper-bound}]
By \cref{lemma:deg-growth-invariant-under-iteration},
replacing $f$ by certain iterate, we may assume that $f^*|_{\N^1(X)_\bR}=\id+N$ is unipotent with $N^{2r}\neq 0$ and $N^{2r+1} = 0$.
In other words, the maximum size of Jordan blocks of $f^*|_{\N^1(X)_\bR}$ is $2r+1$.
It is trivial if $r=0$.
So we assume further that $r\in\bZ_{>0}$ so that \cref{assum:XfN} is satisfied.
By \cref{thm:positivity}, one has $r\le d-1$.

By \cref{prop:zero-case} \ref{prop:zero-case-3}, $N^{2r}H_X \cdot H_X^{d-1}>0$.
Hence 
\begin{align*}
\deg_1(f^n) = (f^n)^*H_X \cdot H_X^{d-1} = \sum_{i=0}^{2r} \binom{n}{i} N^i H_X \cdot H_X^{d-1} \sim C n^{2r},
\end{align*}
for some positive constant $C$, which proves
\cref{cor:degree-growth-upper-bound}.
\end{proof}

\begin{remark}
Similarly, it is not hard to see that
\[
\deg_i(f^n) = (f^n)^*H_X^i \cdot H_X^{d-i} = \Bigg(\sum_{j=0}^{2r} \binom{n}{j} N^j H_X\Bigg)^i \cdot H_X^{d-i} = O(n^{2ir}).
\]
On the other hand, since $\deg_i(f^n) = \deg_{d-i}(f^{-n})$, we also have $\deg_i(f^n) = O(n^{2(d-i)r})$.
Combining them together yields that $\deg_i(f^n) = O(n^{2(d-1)\min (i, d-i)})$, which does not seem to be optimal when $2\le i\le d-2$ (cf. \cite[Remark~4.1]{DLOZ22}).
\end{remark}

\section{Rank of matrices associated with restricted partitions}
\label{sec rank}

In this section we discuss restricted partitions, which are classic objects in combinatorics and play important roles in the proof of \cref{thm:plov}. 

\subsection{Restricted partitions}
\label{subsec partition}

Let $k,d\in \bZ_{>0}$ be fixed positive integers.
For any $n\in \bN$, we call $\lambda = (\lambda_1,\ldots,\lambda_d)$ a {\it restricted partition} of $n$ into at most $d$ parts with each part at most $k$, if $\sum_{j=1}^d \lambda_j = n$ with $\lambda_j\in \bN$ and $k\ge \lambda_1\ge \lambda_2\ge \cdots \lambda_d \ge 0$.
Denote by $P(k, d, n)$ the set of these restricted partitions $\lambda$ and $p(k,d,n)$ the cardinality of $P(k,d,n)$.
Clearly, $p(k,d,n)=0$ if $n>dk$.

The numbers $p(k,d,n)$ of these restricted partitions satisfy the following equation:
\[
\sum_{n=0}^{dk} p(k,d,n) q^n = \binom{d+k}{d}_q = \binom{d+k}{k}_q,
\]
where $( \ )_q$ denotes the $q$-binomial coefficient or Gaussian binomial coefficient.
This, in particular, implies the following properties (see \cite[Chapter~3]{Andrews98} for details):
\begin{itemize}
\item Symmetry in $k$ and $d$: $p(k,d,n) = p(d,k,n)$;
\item Symmetry in $n$: $p(k,d,n) = p(k,d,dk-n)$;
\item Recurrence relation: $p(k,d,n) = p(k,d-1,n) + p(k-1,d,n-d)$.
\end{itemize}
Another highly nontrivial property in combinatorics is the unimodality.
\begin{itemize}
\item Unimodality: 
\begin{align}\label{eq unimod}
\begin{cases}
p(k,d,n) \le p(k,d,n+1), &\text{if } 0 \le n < dk/2;\\
p(k,d,n) \ge p(k,d,n+1), &\text{if } dk/2 \le n < dk.
\end{cases}
\end{align}
\end{itemize}

It is also useful to keep track of the number of times that a particular integer occurs as a part.
Therefore, we often alternatively write a restricted partition $\lambda=(\lambda_1,\ldots,\lambda_d)$ of $n$ as
\[
\lambda = (k^{e_k},\ldots,1^{e_1},0^{e_0}),
\]
where exactly $e_i$ of the parts $\lambda_j$ are equal to $i$.
Note that 
\[
e_i\in \bN, \ \sum_{i=0}^k e_i = d, \text{ and } \sum_{i=0}^k ie_i = n.
\]

\subsection{Weighted incidence matrices associated with restricted partitions}
\label{subsec matrix}

Let $k,d$ be fixed positive integers.
Let $\mu= (k^{e_k},\ldots,1^{e_1},0^{e_0})\in P(k, d, n-1)$ be a restricted partition of $n-1\in \bN$.
For $ 0\le i\le k-1$, denote
\[
\mu(i) \coloneqq (k^{e_k},\ldots,(i+1)^{e_{i+1}+1},i^{e_i-1},\ldots,0^{e_0}).
\]
Namely, $\mu(i)$ is obtained by replacing one part $i$ in $\mu$ by $i+1$.
Then $\mu(i)\in P(k, d, n)$ is a restricted partition of $n$ as long as $e_i>0$.

To encode the relation between $P(k,d,n-1)$ and $P(k,d,n)$, we define a {\it weighted incidence matrix}
\[
A_{k, d, n} \in \Mat_{p(k,d,n-1) \times p(k,d,n)}(\bN)
\]
as follows.
First write $A_{k, d, n} = (a_{\mu, \lambda})_{\mu, \lambda}$ with $\mu\in P(k,d,n-1)$ and $\lambda\in P(k,d,n)$.
Then we define
\[
a_{\mu, \lambda} \coloneqq
\begin{cases}
e_i, &\textrm{if } \mu= (k^{e_k},\ldots,1^{e_1},0^{e_0}) \textrm{ and } \lambda=\mu(i), \\
0, &\text{otherwise}.
\end{cases}
\]
Namely, $a_{\mu, \lambda}$ is the number of ways to obtain $\lambda$ from $\mu$ by increasing exactly one part by $1$ (and reordering if necessary). 
Below we give two examples of such matrices with $k=4$, $d=3$, and $n\in\{6,7\}$.
\[
A_{4,3,6}=
\begin{pmatrix}
1 & 1 & 0 & 0 & 0 \\
1 & 0 & 1 & 1 & 0 \\
0 & 1 & 0 & 2 & 0 \\
0 & 0 & 0 & 2 & 1
\end{pmatrix},
\quad
A_{4,3,7}=
\begin{pmatrix}
1 & 1 & 0 & 0 \\
0 & 2 & 0 & 0 \\
2 & 0 & 1 & 0 \\
0 & 1 & 1 & 1 \\
0 & 0 & 0 & 3 \\
\end{pmatrix}.
\]
Here we adopt the canonical (a.k.a. graded reverse lexicographic) ordering of restricted partitions;
for instance, $P(4,3,6)=\{(4,2,0), (4,1,1), (3,3,0), (3,2,1), (2,2,2)\}$.
It is easy to check that $A_{4,3,6}A_{4,3,7}$ is an invertible matrix and $\rank A_{4,3,6}=\rank A_{4,3,7}=4$. We will see soon in \S\ref{subsec rank} that this is not a coincidence.

\subsection{Rank of \texorpdfstring{$A_{k, d, n}$}{Akdn} via a representation of \texorpdfstring{$\mathfrak{sl}_2(\bC)$}{sl2C}}
\label{subsec rank}

In this subsection, we compute the rank of the weighted incidence matrix $A_{k, d, n}$ defined in \S\ref{subsec matrix}, which is crucial in the proof of \cref{thm:plov}.

\begin{theorem}
\label{thm:full-rank}
Fix positive integers $k,d$ and let $n\in \bN$ be a nonnegative integer.
\begin{enumerate}[label=\emph{(\arabic*)}, ref=(\arabic*)]
\item \label{thm:full-rank-1} If $0\le n < dk/2$, then the product matrix
\[
A_{k,d, n+1}A_{k,d, n+2}\cdots A_{k,d, dk-n} \in \Mat_{p(k,d,n) \times p(k,d,dk-n)}(\bQ)
\]
is invertible (note that $p(k,d,n)=p(k,d,dk-n)$ by symmetry).

\item \label{thm:full-rank-2} If $1\leq n\leq dk$, then the matrix $A_{k, d, n}$ is of full rank, i.e.,
\[
\rank A_{k, d, n} = \min (p(k, d, n-1), p(k, d, n)).
\]
To be more precise,
\[
\rank A_{k, d, n}=
\begin{cases}
p({k, d, n-1}), &\text{if } 1\leq n\leq \lceil dk/2 \rceil, \\
p({k, d, n}), &\text{if } \lfloor dk/2 \rfloor < n\leq dk.
\end{cases}
\]
\end{enumerate}
\end{theorem}
\begin{proof}
Consider the Lie algebra $\mathfrak{sl}_2(\bC)$ generated by
\[
\mathrm H=\begin{pmatrix} 1 & 0\\0 & -1\end{pmatrix}, \quad 
\mathrm X=\begin{pmatrix} 0 & 1\\
0& 0\end{pmatrix}, \quad
\mathrm Y=\begin{pmatrix}0 & 0\\
1& 0\end{pmatrix}.
\]
Let $W_{k}=\bigoplus_{i=0}^{k}\bC x_i$ be the $(k+1)$-dimensional irreducible representation of $\mathfrak{sl}_2(\bC)$, where for each $0\leq i\leq k$, $x_i$ is an eigenvector of $\mathrm H$ with eigenvalue $k-2i$.
Also, rescaling the eigenvectors $x_i$, we may assume that $\mathrm Y(x_i)=x_{i+1}$ for each $0\leq i\leq k-1$ and $\mathrm Y(x_k)=0$.

We then consider the induced representation $\Sym^d W_{k}$ of $\mathfrak{sl}_2(\bC)$.
Since $\mathrm H$ is semisimple, we have an $\mathrm H$-eigendecomposition
\[
\Sym^d W_{k}=\bigoplus_{n=0}^{dk}V^{k, d}_{dk-2n},
\]
where each $V^{k, d}_{dk-2n}$ is a $\bC$-vector space generated by the following basis of monomials
\begin{align}
\label{basis of V}
\Bigg\{x_0^{e_0}x_1^{e_1}\dots x_k^{e_k} : \sum_{i=0}^k ie_i=n, \, \sum_{i=0}^k e_i=d, \, e_i\in \bN \text{ for } 0\leq i\leq k\Bigg\}.
\end{align}
Indeed, $V^{k, d}_{dk-2n}$ is the eigenspace of the induced $\mathrm H$-action with eigenvalue $dk-2n$, as 
\begin{align*}
\mathrm H(x_0^{e_0}x_1^{e_1}\dots x_k^{e_k}) = \bigg(\sum_{i=0}^{k} e_i(k-2i)\bigg) x_0^{e_0}x_1^{e_1}\dots x_k^{e_k}.
\end{align*}
Also, the induced action
\[
\mathrm Y\colon V^{k, d}_{dk-2(n-1)}\lra V^{k, d}_{dk-2n}
\]
is exactly given by the Leibniz rule
\begin{align}
\label{eq:leibniz}
\begin{split}
\mathrm Y(x_0^{e_0}x_1^{e_1}\dots x_k^{e_k})={}& e_0 \, x_0^{e_0-1}x_1^{e_1+1}\dots x_k^{e_k}+\dots + e_{k-1} \, x_0^{e_0}\dots x_{k-1}^{e_{k-1}-1} x_k^{e_k+1} \\
={}&\sum_{i=0}^{k-1} e_{i} \, x_0^{e_0}x_1^{e_1}\dots x_{i}^{e_i-1}x_{i+1}^{e_{i+1}+1} \dots x_k^{e_k}.
\end{split}
\end{align}
Here in the last sum, if $e_{i}=0$, then $x_0^{e_0}x_1^{e_1}\dots x_{i}^{e_i-1}x_{i+1}^{e_{i+1}+1} \dots x_k^{e_k}$ is no longer a monomial.
But we can formally set $e_{i} \, x_0^{e_0}x_1^{e_1}\dots x_{i}^{e_i-1}x_{i+1}^{e_{i+1}+1} \dots x_k^{e_k}=0$ so that the sum is still meaningful.


Note that for each $0\le n\le dk$, our $V^{k, d}_{dk-2n}$ is canonically isomorphic with the $\bC$-vector space $\bC^{p(k,d,n)}$ formally generated by the finite set $P(k, d, n)$ of restricted partitions.
In fact, the basis \eqref{basis of V} of $V^{k, d}_{dk-2n}$ has a natural $1$-$1$ correspondence to $P(k, d, n)$ as follows:
\begin{align*}
\mathbf{x}^\lambda \coloneqq x_0^{e_0}x_1^{e_1}\dots x_k^{e_k}\longleftrightarrow \lambda = (k^{e_k}, \dots, 1^{e_1}, 0^{e_0}).
\end{align*}
Now for any $\mu=(k^{e_k}, \dots, 1^{e_1}, 0^{e_0})\in P(k, d, n-1)$, \cref{eq:leibniz} just becomes
\begin{align*}
\mathrm Y(\mathbf{x}^\mu) = \sum_{\substack{0\leq i\leq k-1;\\[1pt] e_i>0}} e_i \, \mathbf{x}^{\mu(i)} = \sum_{\lambda\in P(k,d,n)} a_{\mu,\lambda} \, \mathbf{x}^{\lambda},
\end{align*}
where $\mu(i)$ and $a_{\mu, \lambda}$ are defined in  \S\ref{subsec matrix}.
In this way,
we have checked that
the following diagram
\begin{equation*}
\begin{tikzcd}
V^{k, d}_{dk-2(n-1)} \arrow[rr, "\mathrm Y"] \arrow[d, "\isom"] & & V^{k, d}_{dk-2n} \arrow[d, "\isom"] \\
\bC^{p(k,d,n-1)} \arrow[rr, "A_{k,d,n}^\sT \cdot \bullet"'] & & \bC^{p(k,d,n)}
\end{tikzcd}
\end{equation*}
commutes, where $A_{k, d, n}^\sT\in \Mat_{p(k,d,n) \times p(k,d,n-1)}(\bN)$ is the transpose of $A_{k, d, n}$ defined in \S\ref{subsec matrix} and the bottom linear map is given by multiplying $A_{k, d, n}^\sT$ with coordinate vectors in columnar form.

By standard knowledge of representation theory of $\mathfrak{sl}_2(\bC)$ (see, e.g., \cite[pp. 118--120]{GH78} or \cite[Theorem~IV.4]{Serre87}), we know that for any $0\leq n < dk/2$, the composite
\[
\mathrm{Y}^{dk-2n}\colon V^{k, d}_{dk-2n}\lra V^{k, d}_{2n-dk}
\]
is an isomorphism.
Hence $A_{k,d,dk-n}^\sT \cdots A_{k,d, n+2}^\sT A_{k,d, n+1}^\sT$ is invertible and Assertion~\ref{thm:full-rank-1} follows.
In particular, we also have that
$\mathrm Y\colon V^{k, d}_{dk-2(n-1)}\to V^{k, d}_{dk-2n}$ is injective if $1\leq n\leq \lceil dk/2 \rceil$ and surjective if $\lfloor dk/2 \rfloor < n \leq dk$.
This concludes the proof.
\end{proof}

\begin{remark}
\label{rmk:unimodality}
According to the proof of \cref{thm:full-rank}, we actually recover the unimodality of $p(k, d, n)$ as in \cref{eq unimod} by the injectivity/surjectivity of $\mathrm Y\colon V^{k, d}_{dk-2(n-1)}\to V^{k, d}_{dk-2n}$.
This is very similar to the idea of Proctor \cite{Proctor82}, who first used $\mathfrak{sl}_2(\bC)$-representation to show such unimodality (which circumvents the hard Lefschetz theorem in Stanley's proof \cite{Stanley82}).
In fact, our idea of the proof 
of \cref{thm:full-rank} originates from \cite{Proctor82}.

But we should emphasize that our construction is different from Proctor's.
The matrices
in Proctor's construction are the $(0,1)$-matrices obtained by replacing all nonzero entries of our $A_{k,d, n}^\sT$ by $1$ (see \cite[\S{4}]{Proctor82}).
Equivalently, his matrices are also obtained by multiplication by a hyperplane section of the Grassmannian $G\coloneqq G(d,d+k)$ (see \cite[Proof of Proposition~9.10]{Stanley82});
however, as far as we know, our matrices may not be obtained in this geometric way and hence they only form a correspondence in $\CH^{dk+1}(G\times G)_\bQ$.
The key point here is that the matrices representing $\mathrm Y$ in our construction are exactly $A_{k,d, n}^\sT$, so we can get interesting properties for these matrices, which is crucial to the proof of our \cref{thm:plov}.
\end{remark}

\begin{remark}
On the other hand, Botong Wang informs us that these matrices $A_{k,d, n}^\sT$ also represent the Lefschetz operators on cohomology of the symmetric product $\Sym^d(\bP^k) \coloneqq (\bP^k)^d/\mathfrak{S}_d$.
Let us reproduce his argument.
Cohomology ring of $\Sym^d(\bP^k)$ is isomorphic to the ring of symmetric polynomials in $d$-variables $x_1,\dots,x_d$ of degree at most $dk$.
For any $1\leq n \leq dk$ and any partition $\lambda = (\lambda_1, \dots, \lambda_d) \in P(k,d,n)$, denote the corresponding monomial symmetric polynomial by
\[
m_\lambda \coloneqq \sum_{\alpha} \prod_{i=1}^d x_i^{\alpha_i},
\]
where the sum is taken over all \emph{distinct} permutations $\alpha = (\alpha_1, \dots, \alpha_d)$ of $\lambda = (\lambda_1, \dots, \lambda_d)$.
It is well known that the set $\{m_\lambda : \lambda\in P(k,d,n)\}$ is a $\bQ$-basis of $H^{2n}(\Sym^d(\bP^k), \bQ)$.
Define
\[
\widetilde{m}_\lambda \coloneqq \frac{e_0!e_1!\cdots e_k!}{d!} m_\lambda, \text{ if } \lambda = (k^{e_k}, \dots, 1^{e_1}, 0^{e_0}).
\]
Then one can verify that the multiplication by the ample divisor $m_{(1,0^{d-1})} = \sum_{i=1}^d x_i$ on $H^{2n-2}(\Sym^d(\bP^k), \bQ)$ under the weighted basis $\{\widetilde{m}_\mu : \mu \in P(k,d,n-1)\}$ is exactly represented by our matrix $A_{k,d,n}^\sT$.
\end{remark}

\section{An upper bound for polynomial volume growth}
\label{sec proofs plov}

\subsection{Proofs of Theorem \ref{thm:plov} and Corollary \ref{cor:plov}}
\label{subsec proofs plov}

\begin{proof}[Proof of \cref{thm:plov}]
Replacing $f$ by certain iterate, we may assume that $f^*|_{\N^1(X)_\bR} = \id + N$ is unipotent (see \cite[Lemma~2.6]{Hu-GK-AV}). By the definition of $k$, $N^k\neq 0$ and $N^{k+1}=0$.
Then thanks to \cref{lemma:plov}, it suffices to show that
\[
(N^{k}H_X)^{e_k} \cdots (N H_X)^{e_1} \cdot H_X^{e_0} = 0,
\]
whenever $\sum_{i=0}^k e_i=d$ and $\sum_{i=0}^k ie_i > dk/2$ with each $e_i\in \bN$. When $k=0$ this is trivial, so we may assume that (the even) $k>0$ (and hence $d\ge 2$).
It suffices to prove the following.

\begin{claim}\label{claim final}
Keep the setting in \cref{thm:plov}. Fix any divisor $D$ on $X$.
For each restricted partition $\lambda=(k^{e_k},\ldots,1^{e_1},0^{e_0})\in P(k,d,n)$ of $n\in \bN$, denote
\[
v_{\lambda} \coloneqq (N^{k}D)^{e_k} \cdots (ND)^{e_1} \cdot D^{e_0}\in \bR.
\]
Then $v_{\lambda}=0$ for any $\lambda\in P(k, d, n)$ and any $dk/2<n\leq dk$.
\end{claim}

We will prove \cref{claim final} by backward induction on $n$.
First we consider the case $n=dk$. The only partition in $P(k, d, dk)$ is $(k^d, (k-1)^0, \dots, 0^0)$. We consider the restricted partition $(k^{d-1}, (k-1)^1, (k-2)^0, \dots, 0^0)$.
Recall that $f^*|_{\N^1(X)_\bR} = \id + N$, $N^k\neq 0$, and $N^{k+1}=0$.
Since $f$ is an automorphism of $X$, the projection formula asserts that
\begin{align*}
v_{(k^{d-1}, (k-1)^1, (k-2)^0, \dots, 0^0)} &= (N^{k}D)^{d-1} \cdot N^{k-1}D\\
&= (f^*(N^{k}D))^{d-1} \cdot f^*(N^{k-1}D)\\
&= (N^kD)^{d-1} \cdot (N^{k-1}D + N^{k}D)\\
&= v_{(k^{d-1}, (k-1)^1, (k-2)^0, \dots, 0^0)}+v_{(k^d, (k-1)^0, \dots, 0^0)}.
\end{align*}
This implies that $v_{(k^d, (k-1)^0, \dots, 0^0)}=0$.

From now on, fix an $n$ such that $dk/2<n<dk$.
We assume that 
$v_{\lambda}=0$ for any $\lambda\in P(k, d, m)$ and any $n< m\leq dk$ as the inductive hypothesis.
We will prove \cref{claim final} for $n$.

Given any $\mu=(k^{e_k},\ldots,1^{e_1},0^{e_0})\in P(k,d,n-1)$, recall that for any $0\le i\le k-1$,
\[
\mu(i) \coloneqq (k^{e_k},\ldots,(i+1)^{e_{i+1}+1},i^{e_i-1},\ldots,0^{e_0})\in P(k,d,n)
\]
is a restricted partition of $n$ as long as $e_i>0$. 
Similarly, as $f$ is an automorphism of $X$, the projection formula asserts that
\begin{align*}
v_{\mu} &= (N^{k}D)^{e_k} \cdot (N^{k-1}D)^{e_{k-1}} \cdots D^{e_0} \\
&= (f^*(N^{k}D))^{e_k} \cdot (f^*(N^{k-1}D))^{e_{k-1}} \cdots (f^*D)^{e_0} \\
&= (N^kD)^{e_k} \cdot (N^{k-1}D + N^{k}D)^{e_{k-1}} \cdots (D + ND)^{e_0} \\
&= v_\mu + \sum_{\substack{0\leq i\leq k-1;\\[1pt] e_i>0}} e_i v_{\mu(i)} + \sum_\xi c_\xi v_\xi,
\end{align*}
where the last sum is taken over some restricted partitions $\xi\in \bigcup_{m\geq n+1}P(k, d, m)$ with certain coefficients $c_\xi$ (though it is not so important by induction, the precise range of $m$ is $n+1\leq m \leq n-1+e_0+\dots+e_{k-1} \le dk$).
Now, by the inductive hypothesis, all such $v_\xi=0$. 
In summary, for each $\mu \in P(k,d,n-1)$, we have the following
linear equation
\[
\sum_{\substack{0\leq i\leq k-1;\\[1pt] e_i>0}} e_i v_{\mu(i)} = 0.
\]
As $\mu$ runs over $P(k,d,n-1)$ and $i$ is allowed to vary, $\mu(i)$ runs over $P(k,d,n)$.
We thus get a homogeneous system of linear equations as follows:
\[
A_{k,d,n}(v_{\lambda})_\lambda = (a_{\mu, \lambda})_{\mu, \lambda} (v_{\lambda})_\lambda=0,
\]
where $A_{k,d,n}$ is the weighted incidence matrix defined in \S\ref{subsec matrix}, $\mu\in P(k,d,n-1)$, $\lambda\in P(k,d,n)$, and $(v_{\lambda})_\lambda$ is a column vector in $\Mat_{p(k,d,n) \times 1}(\bR)$.
As $n>dk/2$, by \cref{thm:full-rank}~\ref{thm:full-rank-2}, we see that $v_{\lambda} = 0$ for all $\lambda\in P(k,d,n)$.
This concludes the proof of \cref{claim final} by induction.
\end{proof}

\begin{proof}[Proof of \cref{cor:plov}]
It follows readily from \cref{thm:plov,cor:degree-growth-upper-bound}.
\end{proof}

\begin{remark}
\label{rmk:Kahler-singular}
Our proof shows that a K\"ahler version of \cref{cor:plov} also holds.
Indeed, one just needs to replace the ample divisor $H_X$ by a K\"ahler class $\omega_X$, $\N^1(X)_\bR$ by $H^{1,1}(X,\bR)$, etc., so that the problem is reduced to solving the same homogeneous systems of linear equations, via the intersection theory of Kähler classes.
\end{remark}

\begin{remark}
\label{rmk:conj-equality}
Recall that in \cref{rmk:conj-lower-bound} we asked whether $\plov(f)\geq d+r(r+1)$, where $r=k/2$.
When $k=2(d-1)$ is maximal, this inequality becomes $\plov(f)\geq d^2$.
Therefore, combining with our \cref{cor:plov}, we wonder whether $\plov(f)=d^2$ once $k=2(d-1)$.
\end{remark}

\subsection{Further discussions and examples}
\label{subsec examples}

In this subsection, we discuss the optimality of \cref{thm:plov} and possible extensions.
First, we can slightly improve the upper bound in \cref{thm:plov} by a constant $1$ for special couples $(k, d)$.
Note that $k$ is always an even nonnegative integer by \cite[Lemma~6.12]{Keeler00} and at most $2d-2$ by \cref{cor:degree-growth-upper-bound}.

\begin{proposition}\label{prop improve -1}
With the same assumption as in \cref{thm:plov}, suppose further that $(k, d)$ is among one of the following couples:
\begin{align*}
{}& (2, d) \text{ with } d \text{ odd},\\
{}&  (6,5), (6,7), (6,9), (6,11), (6,13), (10,7). 
\end{align*}
Then we have
\[
\plov(f) \le (k/2+1)d-1.
\]
\end{proposition}
\begin{proof}
Recall that in \cref{thm:full-rank}~\ref{thm:full-rank-2}, we have shown that all matrices $A_{k,d,n}$ have full column rank $p(k,d,n)$ if $n>dk/2$; moreover, by \cref{lemma:plov}, this implies the upper bound $(k/2+1)d$ in \cref{thm:plov}.
Now, it suffices to show that the matrix $A_{k, d, dk/2}$ also has full column rank $p(k, d, dk/2)$ under our assumption.
Again by \cref{thm:full-rank}~\ref{thm:full-rank-2}, this is equivalent to 
\begin{align}\label{eq break strict unimod}
p(k, d, dk/2-1)=p(k, d, dk/2).
\end{align}

If $(k,d)=(2, d)$ with $d$ odd, it is easy to check that $p(2, d, d-1)=p(2, d, d)=(d+1)/2$. For other couples in our list, one can check that \cref{eq break strict unimod} holds by hand or by a computer program (e.g., the \textsf{IntegerPartitions} function in Mathematica).
\end{proof}

In particular, when $k=2$, \cref{thm:plov} and \cref{prop improve -1} can be summarized as the following corollary, which actually inspires (and convinces) us to work out the general case.

\begin{corollary}
\label{cor:k=2}
Let $X$ be a normal projective variety of dimension $d$ over $\bk$ and $f$ an automorphism of $X$.
Suppose that $f^*|_{\N^1(X)_\bR}$ is quasi-unipotent and the maximum size of Jordan blocks of $f^*|_{\N^1(X)_\bR}$ is $3$.
Then we have
\[
\plov(f) \le 2\lfloor \frac{d}{2}\rfloor +d= 
\begin{cases}
2d, & \text{ if } d \text{ is even}, \\
2d-1, & \text{ if } d \text{ is odd}.
\end{cases}
\]
\end{corollary}
This result was announced in \cite[Theorem~3.5]{Hu-GK-AV} and the upper bound is optimal by \cite[Remark~5.3]{Hu-GK-AV}.
During the preparation of this paper, we were informed by the authors of \cite{LOZ} that they also obtained \cref{cor:k=2} using dynamical filtrations (see \cite[Theorem~4.2]{LOZ}).

The following example shows that the upper bound in \cref{thm:plov} is optimal
as long as $k/2+1$ divides $d$. This example is a modification of \cite[Example~6.4]{LOZ} based on \cite{Hu-GK-AV}.

\begin{example}\label{ex optimal}
Fix a positive integer $d\geq 2$ and a positive even integer $2\le k\le 2(d-1)$.
Denote $d_0=k/2+1$ and write $d=m_0d_0+r_0$ with $m_0\ge 1$ and $0\le r_0<d_0$.
We shall illustrate when the upper bound $(k/2+1)d$ in \cref{thm:plov} is optimal using abelian varieties.

Let $E$ be an elliptic curve.
Let $X=E^d$ and $f$ an automorphism of $X$ which is defined coordinately by the matrix $J_{1,r_0}\oplus J_{1,d_0}^{\oplus m_0}$, where $J_{1,n}$ is a standard Jordan block with eigenvalue $1$ and of size $n$ (note that in \cite[Example~6.4]{LOZ}, $f$ is defined by $J_{1,d}$).
Then by \cite[Theorem~1.5]{Hu-GK-AV}, the pullback $f^*|_{H_\et^{1}(X, \bQ_\ell)}$ of $f$ on the first $\ell$-adic \'etale cohomology $H_\et^{1}(X, \bQ_\ell)$, or equivalently, the induced map $V_\ell(f)$ of $f$ on the $\ell$-adic Tate space $V_\ell(X)$, is given by
\[
(J_{1,r_0}\oplus J_{1,d_0}^{\oplus m_0}) \oplus (J_{1,r_0}\oplus J_{1,d_0}^{\oplus m_0}).
\]
It follows from \cite[Theorem~1.1]{Hu-GK-AV} that $\plov(f)=m_0d_0^2+r_0^2$.
(Here for readers preferring to work over the complex number field, it is more convenient to replace $H_\et^{1}(X, \bQ_\ell)$ by $H^{1,0}(X)$ so that $f^*|_{H^{1,0}(X)}$ is given by $J_{1,r_0}\oplus J_{1,d_0}^{\oplus m_0}$, as in \cite{LOZ}.)

On the other hand, by our construction and \cite[Theorem~1.12]{Hu-GK-AV}, the maximum size of Jordan blocks of $f^*|_{\N^1(X)_\bR}$ is exactly $k+1=2d_0-1$.
Putting together, we obtain that
\[
\plov(f)=m_0d_0^2+r_0^2 \le m_0d_0^2+r_0d_0=d_0(m_0d_0+r_0)=(k/2+1)d.
\]
So our upper bound $(k/2+1)d$ is optimal if $r_0=0$, i.e., $(k/2+1) \mid d$, which, in particular, contains the case $\plov(f)\le d^2$ when $k=2d-2$.
\end{example}
 
\begin{remark}
When $k\ge 4$ and $(k/2+1) \nmid d$, our upper bound $(k/2+1)d$ might be improved not just by $1$.

For example, when $(k, d)=(4,4)$, combining \cref{thm:positivity} in the proof of \cref{thm:plov}, we can actually get an optimal upper bound $\plov(f)\leq 10$ instead of $12$. 
We omit the proof since it is not particularly illuminating.
 
For another example, consider the case $(k,d)=(6,5)$.
Our \cref{prop improve -1} asserts that $\plov(f)\le 19$.
However, using our construction 
in \cref{ex optimal} with $(m_0, d_0, r_0)=(1, 4, 1)$, we can only achieve an example of an automorphism $f$ of an abelian $5$-fold with $\plov(f)=17$. 

In order to obtain sharper upper bounds for the polynomial volume growth or the Gelfand--Kirillov dimension, it might be inevitable to invoke more algebro-geometric or noncommutative algebraic tools.
On the other hand, exploring new methods to construct examples with large $\plov(f)$ is also a very interesting problem.
\end{remark}

\section*{Conflict of Interest}
 The authors declare that they have no conflict of interest.

\section*{Acknowledgments}
The authors are grateful to Botong Wang and Shilin Yu for valuable discussions and for sharing their insights on \cref{thm:full-rank}.
They also thank Jason P. Bell, Serge Cantat, Jungkai A. Chen, Laura DeMarco, Lie Fu, Hsueh-Yung Lin, Keiji Oguiso, Zinovy Reichstein, Hehui Wu, and De-Qi Zhang for stimulating conversations.
The authors further thank the referee for helpful comments and suggestions.
The first author acknowledges the hospitality and support of SCMS at Fudan University during his visit in December 2022.
The second author is a member of LMNS, Fudan University.

\bibliographystyle{amsplain}
\bibliography{mybib}

\end{document}